\documentclass[11pt,a4paper]{amsart}

\usepackage{graphicx}
\usepackage{enumerate}
\usepackage{amsmath, amssymb, amsthm, amscd}
\usepackage{xcolor}
\usepackage{caption}
\usepackage{subcaption}
\usepackage{multirow}

\theoremstyle{plain}
\newtheorem{theorem}{Theorem}[section]

\theoremstyle{definition}
\newtheorem{definition}[theorem]{Definition}

\theoremstyle{remark}
\newtheorem{remark}[theorem]{Remark}

\numberwithin{equation}{section}

\newcommand{\D}{\partial}

\newcommand{\Dlt}{\Delta t}
\newcommand{\Dlx}{\Delta x}

\newcommand{\Fr}{\mathrm{Fr}}

\newcommand{\Ma}{\mathrm{Ma}}
\newcommand{\mbb}{\mathbb}

\newcommand{\mcal}{\mathcal}
\newcommand{\CFL}{\mathrm{CFL}}

\newcommand{\Norm}[1]{{\left\vert\kern-0.25ex\left\vert\kern-0.25ex\left\vert #1 
    \right\vert\kern-0.25ex\right\vert\kern-0.25ex\right\vert}}

\newcommand{\rf}{\mathrm{ref}}

\newcommand{\veps}{\varepsilon}

\makeindex             

\begin{document}

\title[Unified AP and Well-balanced Scheme]{A Unified Asymptotic Preserving and
  Well-balanced Scheme for the Euler System with Multiscale Relaxation}       

\author[Arun]{K.~R.~Arun}
\address{School of Mathematics, Indian Institute of Science Education
  and Research Thiruvananthapuram, Thiruvananthapuram 695551, India} 
\email{arun@iisertvm.ac.in}

\author[Krishnan]{M.~Krishnan}
\address{School of Mathematics, Indian Institute of Science Education
  and Research Thiruvananthapuram, Thiruvananthapuram 695551, India} 
\email{meena9916@iisertvm.ac.in}
\thanks{M.~K.\ gratefully acknowledges the INSPIRE Fellowship
  from Department of Science and Technology, Government of India.}

\author[Samantaray]{S.~Samantaray}
\address{Department of Physics, University of Notre Dame, Nieuwland
  Science Hall, Notre Dame, IN 46556, USA} 
\email{ssamanta@nd.edu}

\date{\today}



\begin{abstract}
  The design and analysis of a unified asymptotic preserving (AP) and
  well-balanced scheme for the Euler Equations with gravitational and 
  frictional source terms is presented in this paper. The asymptotic
  behaviour of the Euler system in the limit of zero Mach and Froude
  numbers, and large friction is characterised by an additional scaling
  parameter. Depending on the values of this parameter, the Euler
  system relaxes towards a hyperbolic or a parabolic limit equation.
  Standard Implicit-Explicit Runge-Kutta schemes are incapable of
  switching between these asymptotic regimes. We propose a time
  semi-discretisation to obtain a unified scheme which is AP for 
  the two different limits. A further reformulation of the
  semi-implicit scheme can be recast as a fully-explicit method in
  which the mass update contains both hyperbolic and parabolic
  fluxes. A space-time fully-discrete scheme is derived using a finite
  volume framework. A hydrostatic reconstruction strategy, an
  upwinding of the sources at the interfaces, and a careful choice of
  the central discretisation of the parabolic fluxes are used to
  achieve the well-balancing property for hydrostatic steady
  states. Results of several numerical case studies are presented to
  substantiate the theoretical claims and to verify the robustness of
  the scheme.  
\end{abstract}

\maketitle

\section{Introduction}
\label{sec:Intro}

Hyperbolic and kinetic equations with relaxation are ubiquitous in
physical problems, such as in the theory of gases \cite{Cer88},
non-equilibrium thermodynamics \cite{CIP94, CC39}, and linear and
nonlinear waves \cite{Whi74}, to name but a few. The main feature of
relaxation problems is the occurrence of lower order terms in the
governing equations, with respect to a small parameter $\veps$ known
as the relaxation parameter. The asymptotic limit $\veps \to 0$ is
often singular in the sense that the original governing equations of
the problem approach a system of equations of a different mathematical
and physical nature in the limit. A rigorous and systematic analysis of
relaxation problems is due to Liu and collaborators \cite{CLL94,
  Liu87}, which later lead to the development of the so-called
relaxation schemes by Jin and Xin \cite{JX95}. Depending on the
particular scaling of the relaxation parameter $\veps$, hyperbolic
relaxation models can yield either hyperbolic or viscous conservation
laws as their asymptotic limits. We refer the interested reader to,
e.g.\ \cite{AN96, Bou99, Nat96, Nat99} for a detailed analysis of
relaxation systems. In \cite{BPR17}, a new class of relaxation
problems, termed as multiscale relaxation problems, is introduced
wherein multiple scalings of the stiff terms are identified by various
powers of the relaxation parameter $\veps$ involving an exponent
$\beta \in [0,1]$. When $\beta \in [0,1)$, the asymptotic limit of
the relaxation system is hyperbolic in nature, and the limit $\veps
\to 0$ is known as hyperbolic-to-hyperbolic relaxation. On the other
hand, $\beta=1$ gives a diffusion equation, and the corresponding
limit is called hyperbolic-to-parabolic relaxation. In short, unlike
standard relaxation systems, the asymptotic limit of a multiscale
relaxation problem is also characterised by the scaling parameter
$\beta$.        

The numerical resolution of relaxation problems, or multiscale problems
in general, poses a lot of difficulties. Standard explicit schemes work
well in the macroscopic regime of the stiff relaxation parameter
$(\veps\sim 1)$ but in the microscopic regime $(\veps\sim 0)$, these
schemes encounter prohibitively expensive stability constraints. In
addition, merely satisfying the stability constraints does not
guarantee the accuracy of the scheme in the microscopic regime. It is
well-known from literature, e.g.\ \cite{Jin95}, that explicit
time-stepping schemes suffer from a severe loss of accuracy in the
stiff regime $\veps \to 0$. Jin, in \cite{Jin99}, introduced the
notion of the so-called ``Asymptotic Preserving (AP) schemes'' in the
context of kinetic models for transport in diffusive regimes, to 
tackle the multiscale nature of the problem and other
associated difficulties. Formally, the basic idea behind an AP scheme
can be explained in a general setting as follows. Let $\mcal{P}_{\veps}$
denote a singularly perturbed problem with $\veps$, the perturbation
parameter. Suppose that in the limit as $\veps \to 0$, the solution of
$\mcal{P}_{\veps}$ converges to the solution of a well-posed problem
denoted by  $\mcal{P}_{0}$, called the singular limit or the limit
problem. A numerical scheme for $\mcal{P}_{\veps}$, denoted by
$\mcal{P}_{\veps}^{h}$ with $h$ being a discretisation parameter, is
said to be asymptotic preserving if,    
\begin{enumerate}[(i)]
\item as $\veps \to 0$ the numerical scheme $\mcal{P}_{\veps}^{h}$
  converges to a numerical scheme $\mcal{P}_{0}^{h}$, which is a
  consistent discretisation of the limit system $\mcal{P}_{0}$, and 
\item the stability constraints on the discretisation parameter $h$
  are independent of $\veps$.
\end{enumerate}
Mathematically, the passage $\mcal{P}_\veps\to\mcal{P}_0$ can often be
formulated as a singular limit of the governing equations; see e.g.\ 
\cite{KM81} for the treatment of a singular limit in hydrodynamics. In
such problems, the AP methodology turns out to be a natural choice for
the numerical approximation, in the sense that it respects the singular
limit at a discrete level, i.e.\
$\mcal{P}_\veps^h\to\mcal{P}_0^h$. Furthermore, the AP framework
automatically recognises the singular and non-singular regions in
the flow as well as the transient regions where regime shifts take
place. Therefore, using an AP discretisation for relaxation problems
involving multiple scales is an effective method that drastically 
reduces the computational complexity while simultaneously enhancing 
the accuracy; see \cite{Jin12} for a detailed review.

During the past decade and a half, the Implicit-Explicit (IMEX)
Runge-Kutta (RK) schemes for solving stiff systems of differential
equations have gained a lot of attention; see \cite{BR13, CJR97, DP11,
  JL96, JPT98, JPT00, Kla98, LM08, PR05} and the references 
therein. IMEX time-stepping schemes rely on a stiff/non-stiff
splitting wherein they add a small amount of implicitness in
comparison to a fully-implicit scheme. The compromise is usually very 
optimal since the need to invert large dense matrices for
a fully-implicit scheme can be avoided, while also getting past the
restrictive stability conditions of a fully-explicit scheme. Individually,
hyperbolic-to-hyperbolic and hyperbolic-to-parabolic relaxation
problems have been successfully tackled by the application of IMEX-RK
schemes; see \cite{CJR97, Jin95, JPT98, JPT00, Kla98, PR05}. In the
case of multiscale relaxation problems, since the asymptotic limits
are non-unique and dependent on $\beta$, the design and analysis of
IMEX AP schemes is a more complex and demanding task. First, the
chosen discretisation should be AP for different asymptotic limits of
the given relaxation system. Second, the restrictions on the
discretisation parameters imposed by stability constraints may not be
uniform as different limit systems are obtained corresponding to
different values of $\beta$. In \cite{BPR17}, a unified AP IMEX-RK
scheme is proposed and analysed for the hyperbolic and parabolic
relaxation limits of a Jin-Xin-type relaxation system. More recently,
in \cite{ADP20}, a unified AP scheme has been developed in an IMEX
linear multistep framework.    

Another difficulty in the numerical approximation of hyperbolic
balance laws containing source terms is the appearance of steady state
solutions wherein flux gradients are exactly balanced by source
terms. For example, in the case of the Euler equations of
hydrodynamics, an equilibrium of interest is the hydrostatic state in 
which the pressure gradients are balanced by the gravitational
force. The so-called well-balanced schemes are those schemes that can
maintain such steady states exactly at the discrete level. Most of the
practical problems of interest are small perturbations of steady
states and therefore, the challenge for a well-balanced scheme is to
resolve such perturbations up to machine precision. One can find
various approaches to derive well-balanced schemes for hyperbolic
balance laws in literature; see e.g.\ \cite{ABB+04, Bou04, CLS04,
  FMT11, GPC07, Gos00, GL96, KP07, NPP+06, NXS07, XS05, XS06, Xu02}.  

The goal of the present work is to design, analyse, and implement a
unified AP and well-balanced scheme for the Euler equations of
compressible flows with multiscale relaxation. Specifically, we
consider the Euler system with an isentropic equation of state
containing gravitational and frictional source terms. Under this
setting, the pertinent non-dimensional parameters which characterise
the multiscale nature of the Euler system are the Mach number, the
Froude number, and a scaled friction coefficient. After a scaling of
these numbers with $\veps$, the pressure gradient term is
$\mcal{O}(\veps^{2\beta})$ whereas the gravity and friction terms
scale as $\mcal{O}(\veps^{1+\beta})$. In the diffusive regime
$(\beta=1)$, the Euler system relaxes to a porous medium equation for
the density with a Darcy-type pressure law. On the other hand, under
the hyperbolic scaling, i.e.\ when $\beta \in [0,1)$, it relaxes to a
linear advection equation for the density. Additionally, due to the
presence of both gravity and friction terms, the Euler system
admits nontrivial stationary solutions which necessitate the need to
have a well-balanced numerical scheme. 

In recent years, AP and well-balanced schemes for the Euler equations
have been an active area of research. The most well-explored singular
limit of the Euler equations is invariably the incompressible limit in
which the compressible Euler system approaches its incompressible
counterpart \cite{KM81}. Various AP schemes have been proposed to
approximate the incompressible limit; see, e.g.\ \cite{AS20, BFN20,
  DLV17, DT11, NBA+14} and the references therein. Details of
well-balancing of the sources in the context of the low Mach number
limit of the Euler equations, or the low Froude number limit of the
shallow water equations can be found, e.g.\ in \cite{BAL+14, BLY17,
  TPK20}. In the presence of source terms, the solution in the
asymptotic regime is often in a state of balance, and hence
well-balancing is crucial for AP schemes applied to stiff systems of
balance laws. In addition, since most of the practical problems of
relevance are perturbations of some steady states, well-balancing is a
key also for the transient regimes wherein $\veps$ is small yet not an
infinitesimal. In the present work, the AP property in a multiscale
relaxation problem for the Euler equations is achieved by performing a
semi-implicit time discretisation along the lines of \cite{BPR17}. A
reformulation of the resulting time semi-discrete scheme allows us to
recast it as an AP, fully-explicit method in which the density update
contains both convection and diffusion terms. Subsequently, a
space-time fully-discrete scheme is obtained in a finite volume
framework. Well-balancing property for the hydrostatic steady state is
accomplished through a novel approach wherein the hyperbolic
convective fluxes are approximated by a Rusanov-type approximate
Riemann solver combined with an equilibrium reconstruction
\cite{ABB+04, Bou04, NRT14, NRT15}, and the parabolic fluxes, by a
simple central differencing. The fully-discrete scheme is shown to be
stable under a parabolic CFL condition which becomes less and less
severe as $\veps \to 0$ for $\beta \in [0,1)$; see also \cite{BPR17}
for details.

The rest of this paper is organised as follows. In
Section~\ref{sec:mult-relax-limits}, we present an asymptotic analysis
of the scaled Euler system with gravity and friction and highlight the
two distinguished limit equations. In
Section~\ref{sec:time-semi-discrete}, we derive a semi-discrete in 
time and semi-implicit scheme and its AP fully-explicit
reformulation. In Section~\ref{sec:space-time-fully}, a space-time
fully-discrete scheme is obtained using a finite volume framework. In
order to maintain the steady states at the discrete level or in other
words, to achieve the well-balancing property, the interface fluxes are
calculated using a hydrostatic reconstruction technique and the source
terms are appropriately upwinded. As the mass update is responsible to 
maintain the AP property of the scheme, a careful choice of the
central discretisation for the parabolic mass fluxes is made in order
to retain the AP property as well as the balance. An analysis of the
scheme thus obtained is carried out to establish its consistency,
well-balancing and AP properties. In
Section~\ref{sec:numer-case-stud}, we present the results of numerical
case studies carried out, which not only substantiate the theoretical
claims but also demonstrate the robustness of the new scheme. Finally,
we close the paper in Section~\ref{sec:con} with a few concluding
remarks.     

\section{Multiscale Relaxation Limits of the Euler System}
\label{sec:mult-relax-limits}

We consider the following one-dimensional compressible Euler equations
with gravity and friction in dimensionless variables: 
\begin{align}
  \D_t\rho+\D_x(\rho u) &=0, \label{eq:euler_mas}\\
  \D_t(\rho u)+\D_x(\rho u^2)+\frac{1}{\Ma^2}\D_xP(\rho) &= - 
       \left(\frac{1}{\bar{\mu}}\rho u-\frac{1}{\Fr^2}\rho \partial_x
                                                     \phi\right). \label{eq:euler_mom} 
\end{align}
Here the independent variables are the time $t>0$ and space
$x\in\mbb{R}$, and the dependent variables are the density $\rho>0$ and
the velocity $u\in\mbb{R}$. The pressure is assumed to follow the
Darcy's law  $P(\rho) = \rho^\gamma$, where $\gamma \geq 1$ is a 
constant. The function $\phi$ represents the known gravitational
potential.  The dimensionless numbers $\Ma, \, \Fr \,$ and $\bar{\mu}$
represent the Mach number, the Froude number and the scaled friction
coefficient respectively. They are defined as   
\begin{equation}
    \Ma := \frac{u_{\rf}}{c_{\rf}}, \quad \Fr :=
    \frac{u_{\rf}}{\sqrt{g x_{\rf}}}, \quad \bar{\mu} :=
    \frac{u_{\rf}}{\mu x_{\rf}}, 
\end{equation}
where $u_{\rf}, \ c_{\rf}$, and $x_{\rf}$ are a reference fluid
speed, a reference sound speed, and a reference length
respectively. The constant $g$ is the gravitational constant and
$\mu$ represents the friction coefficient. 

Let us introduce an infinitesimal parameter $\veps\in (0,1]$ and scale
the non-dimensional numbers via
\begin{equation}
  \label{eq:MaFrX_scales}
  \Ma^2 \sim \mcal{O}(\veps^{2\beta}),  \quad
  \Fr^2=\bar{\mu} \sim \mcal{O}(\veps^{1+\beta}).
\end{equation}
Upon scaling, the compressible Euler equations
\eqref{eq:euler_mas}-\eqref{eq:euler_mom} read
\begin{align}
  \D_t\rho+\D_x(\rho u) &=0, \label{eq:euler_nd_mas}\\
  \D_t(\rho u)+\D_x(\rho u^2)+\frac{1}{\veps^{2\beta}}\D_xP(\rho) &=
       -\frac{1}{\veps^{1+\beta}}(\rho u-\rho\D_x\phi). \label{eq:euler_nd_mom} 
\end{align}
The system \eqref{eq:euler_nd_mas}-\eqref{eq:euler_nd_mom} relaxes to
different asymptotic limits depending on the values of the scaling
parameter $\beta$. In the following, we suppose that $\rho \to
\rho_{(0)}$ and $u \to u_{(0)}$ as $\veps \to 0$. When $\beta = 1$, as
$\veps \to 0$, the Euler system
\eqref{eq:euler_nd_mas}-\eqref{eq:euler_nd_mom} relaxes to the
equilibrium system
\begin{align}
  \D_t\rho_{(0)}+\D_x\left(\rho_{(0)}u_{(0)}\right)&=0, \label{eq:euler_mas_beta1} \\
  \rho_{(0)}u_{(0)}&=\rho_{(0)}\D_x\phi-\D_xP\left(\rho_{(0)}\right).
  \label{eq:euler_mom_beta1}  
\end{align}
Eliminating the momentum $\rho_{(0)}u_{(0)}$ between
\eqref{eq:euler_mas_beta1} and \eqref{eq:euler_mom_beta1} gives the
following parabolic porous medium equation: 
\begin{equation}
  \D_t \rho_{(0)} + \D_x \left(\rho_{(0)} \partial_x \phi\right) =
  \D_{xx}P\left(\rho_{(0)}\right). 
    \label{eq:pmeqn}
\end{equation}
Similarly, when $\beta \in [0,1),$ the momentum equation
\eqref{eq:euler_nd_mom} relaxes to
$\rho_{(0)}u_{(0)}=\rho_{(0)}\D_x\phi$, and the mass equation
converges to \eqref{eq:euler_mas_beta1} again. Consequently, we obtain 
the hyperbolic transport equation  
\begin{equation}
  \D_t \rho_{(0)} + \D_x(\rho_{(0)} \D_x \phi) = 0.
  \label{eq:transporteqn}
\end{equation}

The asymptotic analysis carried out above for the compressible Euler
equations \eqref{eq:euler_nd_mas}-\eqref{eq:euler_nd_mom} shows that
the mass conservation equation incorporates the limiting momentum
equation to give us two different limit equations, depending on the
values of $\beta$. In other words, when $\veps\to0$, the sole surviving
dependent variable is the density $\rho_{(0)}$, and the velocity
$u_{(0)}$ can be be explicitly obtained once $\rho_{(0)}$ is
known. Based on this observation, we make the following definition of
a reformulated and reduced limit equation which pertains to the
limiting mass conservation equation in both of the cases considered
above.   
\begin{definition}\label{def:ref_asymp_lim}
  The reformulated and reduced limit equation for the scaled
  compressible Euler equations
  \eqref{eq:euler_nd_mas}-\eqref{eq:euler_nd_mom} as $\veps\to0$ is
  defined as   
  \begin{enumerate}[(i)]
  \item the parabolic porous medium equation \eqref{eq:pmeqn} when
    $\beta = 1$; and  
  \item the hyperbolic transport equation \eqref{eq:transporteqn} when
    $\beta \in [0,1)$.
  \end{enumerate}
\end{definition}

\subsection{Hydrostatic Steady States}
\label{sec:steady-states}

Hydrostatic steady states are particular stationary solutions of the
Euler system \eqref{eq:euler_nd_mas}-\eqref{eq:euler_nd_mom} that
satisfy 
\begin{align}
  u &= 0, \label{eq:hydro_vl} \\
  \frac{1}{\veps^{2\beta}}\D_xP(\rho) &=
                                  \frac{1}{\veps^{1+\beta}}\rho\D_x\phi. \label{eq:hydro_dn}   
\end{align}
In other words, the hydrostatic steady states correspond to vanishing
velocities and the balancing of the pressure gradient and the
gravitational force. The solutions to
\eqref{eq:hydro_vl}-\eqref{eq:hydro_dn} are not unique, and they
depend on the exponent $\gamma$ in the pressure law. Setting
$\gamma=1$ in \eqref{eq:hydro_dn}, and solving the resulting equations
for $\rho$ yields the following isothermal equilibrium solution:
\begin{equation}
  u_e = 0, \quad
  \rho_e(x)=C\exp{\left(\veps^{\beta-1}\phi(x)\right)}. 
  \label{eq:hydro_iso}
\end{equation}
When $\gamma>1$, analogously, we obtain the following isentropic
equilibrium:
\begin{equation}
  u_e = 0, \quad
  \rho_e(x)=\left(\frac{\gamma-1}{\gamma}\veps^{\beta-1}\phi(x)+C)\right)^{\frac{1}{\gamma-1}}.   
  \label{eq:hydro_ise}
\end{equation}
Here $C$ denotes a constant of integration. 

A multiscale relaxation framework was proposed in \cite{BPR17}, where
a time semi-discrete solver was introduced which converges to an
explicit RK discretisation of the correct asymptotic limit
independently of the scaling parameter. This was achieved by using an
implicit treatment of the mass flux and friction terms, combined with
an explicit treatment of the momentum flux terms.Motivated by this
approach, the primary goal of the present work is to develop a time
semi-discretisation for the system
\eqref{eq:euler_nd_mas}-\eqref{eq:euler_nd_mom}, in order to get a 
unified AP scheme for the reformulated and reduced limit equation. In
other words, the numerical scheme should yield a consistent
discretisation of the parabolic limit \eqref{eq:pmeqn} when $\beta=1$
and that of the hyperbolic limit \eqref{eq:transporteqn} when
$\beta\in [0,1)$. Since most of the problems of interest involve
perturbations of steady states such as \eqref{eq:hydro_iso} or
\eqref{eq:hydro_ise}, we want the proposed scheme to be well-balanced
for the hydrostatic steady states in addition to being AP. 

\section{Time Semi-discrete Scheme}
\label{sec:time-semi-discrete}

In this section, a time discretisation of the compressible Euler system
\eqref{eq:euler_nd_mas}-\eqref{eq:euler_nd_mom} is
proposed. Subsequently, following the design of the scheme, we prove
that the proposed time-discretisation relaxes to the correct
asymptotic limit independent of the choice of the scaling used. Let
$0=t^0<t^1<\cdots<t^n<\cdots$ be an increasing sequence of times and 
let $f^n$ denote an approximation to the value of a function $f$ at
time $t^n$. Along the lines of \cite{BPR17}, we design the following
time semi-discretisation for 
\eqref{eq:euler_nd_mas}-\eqref{eq:euler_nd_mom} in which only the
momentum terms are implicit:  
\begin{align}
  \rho^{n+1}&=\rho^n-\Dlt \D_x q^{n+1},\label{eq:tsd_mas} \\
  q^{n+1}&=q^n-\Dlt\D_x\left(\frac{{q^n}^2}{\rho^n}\right)
           -\frac{\Dlt}{\veps^{2\beta}}\D_xP(\rho^n)
           -\frac{\Dlt}{\veps^{1+\beta}}\left(q^{n+1}-\rho^n\D_x\phi\right),
           \label{eq:tsd_mom}  
\end{align}
where $q=\rho u$. As the next step, we perform a reformulation of the
above scheme so that the mass update \eqref{eq:tsd_mas} rewrites as a
perturbation of a discretisation of the reduced limit equation for all
values of $\beta$. To this end, we eliminate $q^{n+1}$ between
\eqref{eq:tsd_mas}-\eqref{eq:tsd_mom}, and recast the resulting
update formulae in the following incremental form:  
\begin{align}
\rho^{n+1}&= \rho^n- \frac{\veps^{1+\beta} \Delta t }{\veps^{1+\beta}
            +\Delta t}\partial_x q^n  +  \frac{\veps^{1+\beta} \Delta
            t^2 }{\veps^{1+\beta} +\Delta t} \partial_{xx}
            \left(\frac{{q^n}^2}{\rho^n}\right) \nonumber\\& \quad \quad +
  \frac{\veps^{1-\beta} \Delta t^2 }{\veps^{1+\beta} + \Delta t}
  \partial_{xx}P(\rho^n) -  \frac{ \Delta t^2 }{\veps^{1+\beta}
  +\Delta t} \partial_{x}(\rho^n\partial_x \phi), \label{eq:tsd_rf_mas}\\
  q^{n+1}&=q^n-\frac{\veps^{1+\beta}\Delta t}{\veps^{1 + \beta} + \Delta
           t} \partial_x \left(\frac{{q^n}^2}{\rho^n}\right)-
           \frac{\veps^{1-\beta}\Delta t}{\veps^{1 + \beta} + \Delta
         t} \partial_x P(\rho^n)-\frac{\Delta t}{\veps^{1 + \beta} +
         \Delta t} (q^{n}-\rho^n \partial_x \phi)). \label{eq:tsd_rf_mom}
\end{align}
\begin{remark}
  It can be easily seen that upon a Taylor expansion, the first
  derivative terms in the above update formulae
  \eqref{eq:tsd_rf_mas}-\eqref{eq:tsd_rf_mom} constitute a hyperbolic
  system whose flux function depends explicitly on $\veps$ and
  $\Dlt$. The eigenvalues of its Jacobian matrix converge to the
  eigenvalues of the Euler system
  \eqref{eq:euler_nd_mas}-\eqref{eq:euler_nd_mom} as $\Dlt\to0$ when
  $\veps$ is fixed. On the other hand, for a fixed $\Dlt$, the
  eigenvalues remain bounded in the limit of
  $\veps\to0$. As a consequence, the stiffness in the stability
  condition can be overcome by using the semi-discrete scheme
  \eqref{eq:tsd_rf_mas}-\eqref{eq:tsd_rf_mom}. It has to be 
  noted that the presence of the second order terms in
  \eqref{eq:tsd_rf_mas}-\eqref{eq:tsd_rf_mom} would impose a stricter
  parabolic stability restriction in a space-time
  discretisation. However, it has been shown in \cite{BPR17} that the
  parabolic stability constraint does not degenerate in the stiff limit
  $\veps\to0$.       
\end{remark}
We designate the updates \eqref{eq:tsd_rf_mas}-\eqref{eq:tsd_rf_mom}
as the reformulated time semi-discrete scheme. Note that the mass
update \eqref{eq:tsd_rf_mas} now contains both hyperbolic and
parabolic terms which are essential to get consistency with the
reduced limit equation for all values of $\beta$. We state the AP
property of the scheme as follows. 
\begin{theorem}\label{thm:tsd_scheme_thm}
  The reformulated time semi-discrete scheme
  \eqref{eq:tsd_rf_mas}-\eqref{eq:tsd_rf_mom} is consistent with the
  Euler system \eqref{eq:euler_nd_mas}-\eqref{eq:euler_nd_mom} away
  from vacuum. Furthermore, it is asymptotically consistent with the
  reformulated and reduced limit equation given in
  Definition~\ref{def:ref_asymp_lim}. In other words, as $\veps \to 
  0$, the mass update \eqref{eq:tsd_rf_mas} yields a consistent time
  semi-discretisation of  
  \begin{enumerate}[(i)]
  \item the parabolic porous medium equation \eqref{eq:pmeqn} for
    $\beta = 1$; and 
  \item the hyperbolic transport equation \eqref{eq:transporteqn} for
    $\beta \in [0,1)$.
  \end{enumerate}
\end{theorem}
\begin{proof}
  The consistency with the Euler equations
  \eqref{eq:euler_nd_mas}-\eqref{eq:euler_nd_mom} follows by Taylor
  expanding the terms on the right hand side of
  \eqref{eq:tsd_rf_mas}-\eqref{eq:tsd_rf_mom}. 
  
  Setting $\beta=1$ in \eqref{eq:tsd_rf_mas}, and letting $\veps\to0$
  we get
  \begin{equation}
    \label{eq:AP_thm_tsd_mas}
    \rho_{(0)}^{n+1}=\rho_{(0)}^n-\Dlt\D_x\left(\rho_{(0)}^n\D_x\phi\right)
    +\Dlt\D_{xx}P\left(\rho^n_{(0)}\right),
  \end{equation}
  which is a consistent semi-discrete approximation of the porous
  medium equation \eqref{eq:pmeqn}. Similarly, when $\beta \in [0,1)$,
  the limit $\veps \to 0$ yields
   \begin{equation}
    \label{eq:AP_thm_tsd_mom}
    \rho_{(0)}^{n+1}=\rho_{(0)}^n-\Dlt\D_x\left(\rho_{(0)}^n\D_x\phi\right),
  \end{equation}
  which is consistent with the transport equation
  \eqref{eq:transporteqn}. 
\end{proof}

\section{Space-time Fully-discrete Scheme and Well-balancing}
\label{sec:space-time-fully}

This section is devoted to the design and analysis of a fully-discrete
version of the time semi-discrete scheme
\eqref{eq:tsd_rf_mas}-\eqref{eq:tsd_rf_mom}. Towards this end, we use
a finite volume approach to approximate the spatial derivatives. In 
order to maintain discrete steady states, the interface fluxes in the
momentum update are modified using a hydrostatic reconstruction
technique. The source terms are appropriately upwinded to serve the
task at hand; see also \cite{NRT14}. Since the mass update is
reformulated by incorporating the updated momentum to achieve the AP
property, it is essential to preserve the discrete steady state in a
discretisation of the mass conservation as well. We make a prudent
choice of a central discretisation to maintain the balance in the
density update.  

Let us recast the semi-discrete scheme
\eqref{eq:tsd_rf_mas}-\eqref{eq:tsd_rf_mom} in the following compact
form: 
\begin{align}
  \frac{U^{n+1}-U^n}{\Delta t}+ c_1\partial_x F(U^n) = c_2\partial_{x}
  G(U^n, \D_xU^n)+c_1 S(U^n),
  \label{eq:tsd_comp}
\end{align}
where the vector of conserved variables $U$, the hyperbolic flux $F$,
the parabolic flux $G$ and the source term $S$ are defined as
\begin{gather}
  U=\begin{pmatrix} \rho\\ q \end{pmatrix},  \quad
  F(U) = \begin{pmatrix} q \\ \frac{{q}^2}{\rho}
    +\frac{P(\rho)}{\veps^{2\beta}}\end{pmatrix}, \label{eq:tsd_comp_defn1}\\
  G(U,\D_xU) = \begin{pmatrix} \partial_{x} \left(\frac{{q}^2}{\rho}\right)+ \frac{\partial_{x}P(\rho)}{\veps^{2\beta}} - 
    \frac{\rho\partial_x \phi}{\veps^{1+\beta}} \\
    0 \end{pmatrix},\quad
  S(U) = \begin{pmatrix} 0 \\  -\frac{q-\rho \partial_x
      \phi}{\veps^{1+\beta}} \end{pmatrix}, \label{eq:tsd_comp_defn2}
\end{gather}
with the shorthands $c_1 =  \frac{\veps^{1+\beta}}{\veps^{1+\beta}
  +\Delta t}$ and $c_2 =  c_1\Delta t$. 

In order to get a fully-discrete scheme, we use a finite volume
framework. The first step is to divide the computational domain
$[a,b]$ into $N$ cells $C_i=[x_{i-1/2},x_{i+1/2}]$ for
$1\leq i \leq N$. For simplicity, we assume that the cells $C_i$ have
an equal length $\Dlx$. The unknown $U_i^n$ is an approximation to the
average of $U$ in the cell $C_i$ at time $t^n$, i.e.\ 
\begin{equation}
  \label{eq:fv_avg}
  U_i^n=\frac{1}{\Dlx}\int_{x_{i-1/2}}^{x_{i+1/2}}U(t^n,x)dx.
\end{equation}

Integrating \eqref{eq:tsd_comp} over the cell $C_i$ yields the
following finite volume discretisation:
\begin{equation}
  \label{eq:stfd_fv}
  \frac{U_i^{n+1}-U_i^n}{\Dlt} + c_1\frac{\mcal{F}^n_{i+1/2} -
    \mcal{F}^n_{i-1/2}}{\Dlx} = c_2 \frac{\mcal{G}^n_{i+1/2} -
    \mcal{G}^n_{i-1/2}}{\Dlx} + c_1 S^n_i,
\end{equation}
where $\mcal{F}_{i+1/2}$ and $\mcal{G}_{i+1/2}$ are,
respectively, approximations to the fluxes $F$ and $G$ at the interface
$x_{i+1/2}$, and $S_i$ is an approximation to the source term
$S$ in the cell $C_i$. It is well-known from literature that the
usual approach of defining $\mcal{F}_{i+1/2}$ as $\mcal{F}(U_i,
U_{i+1})$, where $\mcal{F}$ is a consistent numerical flux of the
homogenous Euler system, and a pointwise calculation of the source
term as $S_i = S(U_i)$ produces large errors near non-constant steady
states. In other words, such a choice does not lead to a well-balanced
scheme. Hence, it becomes crucial to incorporate changes in the
numerical scheme which helps to overcome this challenge.   

The main idea behind using the so-called hydrostatic reconstruction
method \cite{ABB+04,Bou04} to maintain well-balancing involves
designating $\mcal{F}_{i+1/2}$ as
$\mcal{F}(U^-_{i+1/2},U_{i+1/2}^+)$ instead of
using $\mcal{F}(U_i, U_{i+1})$. Here, $U^\pm_{i+1/2}$ are
appropriate reconstructions of the conserved variable $U$ at the
interface $x_{i+1/2}$, to be made precise later. The method also
involves upwinding of the sources at interfaces wherein the source
term $S_i$ is discretised as $S^-_{i+1/2} + S^+_{i-1/2}$ with the
upwind contributions defined as  
\begin{equation}
  \label{eq:usi_defns}
  S^-_{i+1/2} :=\begin{pmatrix} 0 \\
    \frac{P\left(\rho^-_{i+1/2}\right)-P(\rho_i)}{\Delta
      x} \end{pmatrix}, \quad 
  S^+_{i-1/2} := \begin{pmatrix} 0 \\
    \frac{P(\rho_i)-P\left(\rho^+_{i-1/2}\right)}{\Delta
      x} \end{pmatrix}. 
\end{equation}

We will show that concentrating the source terms at the interfaces in this
manner preserves the stationary solutions, provided the interface
values $U^\pm_{i+1/2}$ are defined in such a way so as to take 
into account the source terms. The parabolic fluxes $\mcal{G}_{i+1/2}$
are defined as $\mcal{G}(U_i,U_{i+1})$ using a consistent
numerical flux function $\mcal{G}$. Lastly, since the presence of the
expression $\frac{-q} {\veps^{1+\beta} + \Delta t}$ in the source term
$S(U^n)$, cf.\ \eqref{eq:tsd_comp_defn2}, is a result of rearranging
the implicit momentum equation \eqref{eq:tsd_mom} in order to write it
as an explicit scheme, we will not include it in the reconstruction
step. Binding together all the strategies discussed thus far,
the final scheme takes the form   
\begin{align}
  \label{eq:stfd_scheme_gen}
  \begin{split}
    \frac{U^{n+1}_i - U^n_i}{\Delta t}  + c_1
    \frac{\mcal{F}\left(U^{n,-}_{i+1/2},U_{i+1/2}^{n,+}\right)-
      \mcal{F}\left(U^{n,-}_{i-1/2},U_{i-1/2}^{n,+}\right)}{\Delta
      x} &= c_2 \frac{\mcal{G}(U_i^n,U_{i+1}^n) -
      \mcal{G}(U_{i-1}^n,U_i^n)}{\Delta x} \\&\quad + c_1 \left(
      S^{n,-}_{i+1/2} + S^{n,+}_{i-1/2} \right) +
    c_1\begin{pmatrix} 0  \\ -q_i^n\end{pmatrix}. 
\end{split}
\end{align}

Next, we draw our attention towards well-balancing for hydrostatic
equilibrium solutions. To this end, let us consider the pair
$(U_i^n,\phi_i)$, where $\phi_i$ is some consistent approximation of
the potential $\phi$ at the cell centres $x_i$. Following \cite{Bou04}, we
define a discrete steady state as follows.
\begin{definition} \label{defn:disc_steady_state}
  A sequence $(U_i,\phi_i)_{i\in\mbb{Z}}$ is said to be a discrete
  hydrostatic steady state of the Euler system if
  \begin{equation}
    \label{eq:disc_steady_state}
    q_i=0, \quad
    \mcal{D}\left(\rho_i,\rho_{i+1},\phi_i,\phi_{i+1}\right)=0, \
    \mbox{for all} \ i\in\mbb{Z}, 
  \end{equation}
  where the finite difference operator $\mcal{D}$ gives a consistent
  discretisation of the balance \eqref{eq:hydro_dn}. The numerical
  scheme \eqref{eq:stfd_scheme_gen} is said to well-balanced if
  $(U_i^{n+1},\phi_i) _{i\in\mbb{Z}}$ is a discrete steady state
  whenever $(U_i^n,\phi_i) _{i\in\mbb{Z}}$ is a discrete steady state.  
\end{definition}

For explicit finite volume schemes, the main ingredients for
well-balancing are the hydrostatic reconstruction and the upwinding of
the sources at the interfaces; see e.g.\ \cite{NRT14, NRT15} for an
application of these strategies to the Euler system with gravity and
friction. If $(U_i^n,\phi_i)$ is a discrete steady state, the
hydrostatic reconstruction technique ensures that all the velocity
terms vanish in both the numerical flux function and the source
term. Then the upwinding of the sources at the interfaces makes sure 
that the remaining terms correspond to the second identity in
\eqref{eq:disc_steady_state}. As a result, the updated solution
$(U_i^{n+1},\phi_i)$ boils down to a discrete steady
state. Maintaining the hydrostatic balance in the current problem is
much more complex and complicated in comparison to standard explicit
time-stepping schemes. The occurrence of the parabolic flux terms and
the term $\partial_x (\rho \D_x \phi)$ in the mass equation, cf.\
\eqref{eq:tsd_rf_mas}, arising from the reformulation necessitates the
need to choose an appropriate discretisation for these terms in order
to preserve the well-balancing property also in the mass equation. In
the present work we use the hydrostatic reconstruction and the
upwinding of the sources to achieve a balance between the hyperbolic
flux terms and the source terms in the momentum equation. Once this
balance is achieved, we warily devise a discretisation of the
parabolic terms in the modified mass update to enforce and maintain
the same balance.     

\subsection{Reconstruction of the Interface Values}
\label{sec:rec_interface}

In this subsection, we define the reconstructed values $\rho_{i+1/2}^\pm$
and $u_{i+1/2}^\pm$ of the primitive variables. We omit the
superscript $n$ for convenience since the scheme is now fully explicit
and there is no confusion. In order to define the interface values, we
use the following stationary system:   
\begin{align}
  \begin{split}
    \partial_x(\rho u) &= 0,\\
    \partial_x\left(\rho u^2 +\frac{1}{\veps^{2\beta}}
      P(\rho)\right) &= -\frac{1}{\veps^{1+\beta}}(\rho u-\rho
    \partial_x \phi).
\label{eq:eulerstat} 
\end{split}
\end{align}
We also define a function $\psi$ via
\begin{equation}
  \psi(\rho) := e(\rho) + \frac{P(\rho)}{\rho} = \frac{\gamma}{\gamma
    -1} \rho^{\gamma - 1},\ \text{for } \gamma > 1,
\end{equation}
where $e$ denotes the internal energy function given by
$e^\prime(\rho) := \frac{P(\rho)}{\rho^2}$ for the isentropic pressure
law. In terms of the function $\psi$, the stationary system
\eqref{eq:eulerstat} can be rewritten as  
\begin{align}
    \begin{split}
        \partial_x (\rho u) &= 0,\\
        \partial_x \left( \frac{u^2}{2} +
          \frac{1}{\veps^{2\beta}}\psi(\rho) \right) &=
        -\frac{1}{\veps^{1+\beta}} (u- \partial_x \phi),  
        \label{eq:eulerstatpsi}
    \end{split}
\end{align}
after dividing the second equation in \eqref{eq:eulerstat} by
$\rho$. Now integrating the equations in \eqref{eq:eulerstatpsi} over
the half-cell $[x_i, x_{i+1/2}]$, and assuming that the velocity is
constant, we get  
\begin{align}
    \begin{split}
        u^-_{i+1/2} &= u_i,\\
        \frac{1}{\veps^{2\beta}}\psi(\rho^-_{i+1/2}) -
        \frac{1}{\veps^{2\beta}}\psi(\rho_i) &=
        \frac{1}{\veps^{1+\beta}} (\phi_{i+1/2}-\phi_i). 
    \end{split}
\end{align}
Taking $\phi_{i+1/2} = \min(\phi_i, \phi_{i+1})$, the final
reconstructed interface values take the form    
\begin{align}
    \begin{split}
       \veps^{1-\beta} \psi(\rho_{i+1/2}^-) &= \left[\veps^{1-\beta}
       \psi(\rho_{i})  +(\min(\phi_i, \phi_{i+1})-\phi_i)\right]_+,\\ 
       \veps^{1-\beta} \psi(\rho_{i+1/2}^+) &= \left[\veps^{1-\beta}
       \psi(\rho_{i+1})  +(\min(\phi_i, \phi_{i+1}) - \phi_{i+1})\right]_+.
    \label{eq:ereconstruct}
    \end{split}
\end{align}
Here the truncations $[X]_+:=\max(0,X)$ are present to ensure the
positivity of the reconstructed density. For the isentropic gas law
with $\gamma > 1$, the function $\psi$ is continuous and strictly
increasing, and thus is invertible in $[0, \infty)$. Hence, we can find
the interface values $\rho_{i+1/2}^\pm$ by inverting the
relations in \eqref{eq:ereconstruct}. The reconstruction
\eqref{eq:ereconstruct} is referred to as the E-reconstruction
\cite{NRT14}. A  similar analysis can also be performed using
(\ref{eq:eulerstat}) to get what is known as the P-reconstruction 
\begin{align}
    \begin{split}
       \veps^{1-\beta} P(\rho_{i+1/2}^-) &= \left[\veps^{1-\beta}
         P(\rho_{i}) +\bar{\rho}_{i+1/2}(\min(\phi_i,
         \phi_{i+1})-\phi_i)\right]_+,\\ 
       \veps^{1-\beta} P(\rho_{i+1/2}^+) &= \left[\veps^{1-\beta}
         P(\rho_{i+1})  +\bar{\rho}_{i+1/2}(\min(\phi_i, \phi_{i+1}) -
         \phi_{i+1})\right]_+, 
       \label{eq:reconstruct}
    \end{split}
\end{align}
where we fix $\bar{\rho}_{i+1/2} = \frac{1}{2}(\rho_i +
\rho_{i+1})$. Note that unlike $\psi$, the function $P$ is an
invertible function for $\gamma = 1$ as well.  

We assert that computing the hyperbolic fluxes using the above
interpolated states and concentrating the source term at the
interfaces will ensure the well-balancing of the momentum
equation. Before we attempt to prove this claim, we need to fix an
appropriate discretisation for the mass equation to ensure overall
well-balancing for the resulting scheme.  

\subsection{Well-Balancing for the Mass Conservation Equation} 
\label{sec:wb_mass}

In order to present the basic ideas behind the balance, we consider
again the following modified mass equation of the time semi-discrete
scheme which has both hyperbolic and parabolic flux terms:
\begin{align}
  \label{eq:tsd_mas_reconsider}
  \begin{split}
    \rho^{n+1}&= \rho^n- \frac{\veps^{1+\beta} \Dlt
    }{\veps^{1+\beta}+\Dlt}\D_x q^n  + \frac{\veps^{1+\beta}
      \Dlt^2}{\veps^{1+\beta} +\Dlt}
    \D_{xx}\left(\frac{{q^n}^2}{\rho^n}\right) \\& \quad \quad
    +\frac{\veps^{1-\beta} \Dlt^2 }{\veps^{1+\beta} + \Dlt}
    \D_{xx}P(\rho^n) -  \frac{ \Dlt^2 }{\veps^{1+\beta} +\Dlt}
    \D_{x}(\rho^n\D_x \phi).   
  \end{split}
\end{align}

Since a hydrostatic steady state demands the vanishing of velocity at
all times, we notice that the term $\partial_x q^n$ disappears when we
use the equilibrium reconstruction combined with a consistent
numerical flux for the hyperbolic terms. Analogously, the second order
term $\partial_{xx} (\frac{{q^n}^2}{\rho^n})$ will also be zero when we
use a consistent discretisation. Hence, we now only need to preserve
the balance between the parabolic term $\partial_{xx}P(\rho)$ and the
hyperbolic term $\partial_x(\rho \partial_x \phi)$ in the mass
equation in order to maintain the stationarity of $\rho$. 

This can be accomplished by noting the following discrete form of the
equilibrium solution that was used earlier in the P-reconstruction: 
\begin{equation}
  \veps^{1-\beta}P(\rho_{i+1})-\bar{\rho}_{i+1/2}\phi_{i+1}= 
  \veps^{1-\beta}P(\rho_i)-\bar{\rho}_{i+1/2}\phi_{i}. 
\end{equation}
Therefore we believe that using a central discretisation for the
parabolic flux and an averaged upwind flux of the form  
\begin{equation*}
  \frac{\bar{\rho}^n_{i+1/2}(\phi_{i+1}-\phi_{i})-\bar{\rho}^n_{i-1/2}(\phi_{i}-\phi_{i-1})}{\Dlx^2}   
\end{equation*}
for the term $\D_x(\rho \partial_x \phi)$ will ensure a balance between the
two terms for a hydrostatic solution.  
\begin{remark}
  Note that, using the E-reconstruction to construct the balance can
  lead to the loss of the AP property for an isentropic Euler
  system. This is because it can lead to a wrong diffusion coefficient
  for the parabolic limit equation; see \cite{NRT15} for more
  details. Moreover, the P-reconstruction technique can be employed
  for $\gamma = 1$ as well. Hence in all further mathematical and
  numerical analysis presented, we will only make use of the
  P-reconstruction technique.    
\end{remark}

Using the hydrostatic reconstruction for the hyperbolic fluxes
and the numerical source term, and discretising the mass equation in
the manner described above, the final scheme takes the form 
\begin{align}
\rho^{n+1}_i&=\rho^n_i-c_1\frac{\mcal{F}^\rho (U_{i+1/2}^{n,-},U_{i+1/2}^{n,+})-\mcal{F}^\rho(U_{i-1/2}^{n,-},U_{i-1/2}^{n,+})}{\Dlx} + c_2 \frac{\left(\frac{{q_{i+1}^n}^2}{\rho_{i+1}^n}\right)-\left(\frac{{2{q_{i}^n}^2}}{\rho_{i}^n}\right)+\left(\frac{{q_{i-1}^n}^2}{\rho_{i-1}^n}\right)}{\Dlx ^2}\nonumber\\&\quad+ c_2\frac{P(\rho_{i+1}^n)-2P(\rho_{i}^n)+P(\rho_{i-1}^n)}{\veps^{2\beta}\Dlx^2} -  c_2 \frac{\bar{\rho}^n_{i+1/2}(\phi_{i+1}-\phi_{i})-\bar{\rho}^n_{i-1/2}(\phi_{i}-\phi_{i-1})}{\veps^{1+\beta}\Dlx ^2}, \label{eq:stfd_mas}\\
q^{n+1}_i&=q^n_i- c_1 \frac{\mcal{F}^q (U_{i+1/2}^{n,-},U_{i+1/2}^{n,+})-\mcal{F}^q(U_{i-1/2}^{n,-},U_{i-1/2}^{n,+})}{\Dlx} -c_1 q^{n}_i  + c_1 \frac{P(\rho_{i+1/2}^{n,-})-P(\rho_{i-1/2}^{n,+})}{\veps^{2\beta}\Dlx},\label{eq:stfd_mom}
\end{align}
where $\mcal{F} = \left(\mcal{F}^\rho, \mcal{F}^q\right)^T$ is a
consistent numerical flux of the homogeneous Euler system. In what
follows, we establish the consistency, well-balancing and AP
properties of the space-time fully-discrete scheme
\eqref{eq:stfd_mas}-\eqref{eq:stfd_mom}. 

\begin{theorem} \label{thm:stfd_scheme_thm}
  The fully-discrete scheme \eqref{eq:stfd_mas}-\eqref{eq:stfd_mom}
  \begin{enumerate}[(i)]
  \item is a consistent discretisation of the Euler system
    \eqref{eq:euler_nd_mas}-\eqref{eq:euler_nd_mom} away from vacuum, 
  \item is well-balanced for the hydrostatic steady state in the sense
    of Definition~\ref{defn:disc_steady_state},
  \item is AP for the reformulated and reduced limit equation as $\veps
    \to 0$, independent of $\beta$. 
\end{enumerate}
\end{theorem}

\begin{proof}
  It is straightforward to show that the central discretisations of
  the parabolic fluxes and the term $\D_x(\rho^n\D_x \phi)$ are
  consistent. Hence, in order to prove (i), we need to show only the
  consistency of the hyperbolic numerical fluxes and the source
  term. To this end, we follow the approach of \cite{PS03}. First, we
  show that 
  \begin{equation}
    \label{eq:flux_cons}
    \lim_{U_i^n,U_{i+1}^n \to U, \, \Dlx \to 0} \mcal{F}(U_{i+1/2}^{n,-},U_{i+1/2}^{n,+}) = F(U).
  \end{equation}
  As in \cite{NRT14}, we use a Taylor expansion to see that 
  \begin{equation}
    U_{i+1/2}^{n,-} = U_i^n + \mcal{O}(\Dlx), \quad U_{i+1/2}^{n,+} =
    U_{i+1}^n + \mcal{O}(\Dlx). 
  \end{equation}
  Hence, we have that
  \begin{equation}
    \mcal{F}(U_{i+1/2}^{n,-},U_{i+1/2}^{n,+}) = \mcal{F}(U^n_i, U^n_{i+1}) + \mcal{O}(\Dlx).
  \end{equation}
  Therefore, from the consistency of the numerical flux function
  $\mcal{F}$, in the limit, the consistency of
  $\mcal{F}(U_{i+1/2}^{n,-},U_{i+1/2}^{n,+})$ also follows.  

  Now we have to prove the consistency of the source term
  discretisation, i.e.\ the upwinding of the source term at the
  interfaces. This notion of consistency is defined as per \cite{PS03}
  where we need to show that  
  \begin{equation}
    \lim_{U^n_{i},U^n_{i+1}\to U, \, \Dlx\to
      0}\left\{\left(S_{i+1/2}^{n,-}+S_{i+1/2}^{n,+}\right)-
      \frac{q^n_i}{\veps^{1+\beta}}\right\}=S(U). 
  \end{equation}
  From the definition of the upwind contributions of the source term,
  cf.\ \eqref{eq:usi_defns}, we have that
  \begin{equation}
    \label{eq:usi_terms_sum}
    S_{i+1/2}^{n,-}+S_{i+1/2}^{n,+}= \frac{1}{{\veps^{2\beta}}\Dlx}
    \begin{pmatrix}
      0 \\
      P(\rho_{i+1/2}^{n,-})-P(\rho^n_{i})+P(\rho^n_{i+1})-P(\rho_{i+1/2}^{n,+})\end{pmatrix}. 
  \end{equation}
  For the P-reconstruction, we can expand the above relation using a
  taylor expansion, omitting the positivity preserving truncations (we
  assume the density to be away from vacuum) to yield
  \begin{equation}
    \begin{aligned}
      {\veps^{1-\beta}}P(\rho_{i+1/2}^{n,-})-{\veps^{1-\beta}}P(\rho_{i+1/2}^{n,+})
      &=
      {\veps^{1-\beta}}P(\rho^n_{i})+\bar{\rho}^n_{i+1/2}\left(\frac{\D_x\phi_{i}-|\D_x\phi_{i}|}{2}\right)\Dlx \\ 
      &-{\veps^{1+\beta}}P(\rho^n_{i+1})-\bar{\rho}^n_{i+1/2}\left(-\frac{\D_x\phi_{i+1}+|\D_x\phi_{i+1}|}{2}\right)\Dlx+O(\Dlx^2).
    \end{aligned}
  \end{equation}
  Therefore, in the limit $U^n_{i},U^n_{i+1}\to U, \, \Dlx\to 0$, we have 
  \begin{equation}
    \frac{P(\rho_{i+1/2}^{n,-})-P(\rho_{i+1/2}^{n,+})}{{\veps^{2\beta}}\Dlx}
    = \frac{\rho\D_x\phi}{\veps^{1+\beta}}, 
  \end{equation}
  and the discretisation of the source terms is thus consistent.

To prove (ii), let us assume that $(U_i^n,\phi_i)$ be a discrete
hydrostatic solution. We take this solution in the following form: 
\begin{align}
    \begin{split}
        u^n_i &= 0, \\
        \veps^{1-\beta}P(\rho^n_{i+1})-\bar{\rho}^n_{i+1/2}\phi^n_{i+1}
        &= \veps^{1-\beta}P(\rho_i^n)-\bar{\rho}^n_{i+1/2}\phi_{i}^n. 
    \end{split}
    \label{eq:discrete_stat}
\end{align}
Note that the second equation in the above system defines the finite
difference operator $\mcal{D}$ introduced in
Definition~\ref{defn:disc_steady_state}. Using the expressions
\eqref{eq:reconstruct} for reconstructed states, it is evident that
for a discrete stationary solution with zero velocity
\eqref{eq:discrete_stat}, we have that $U^{n,-}_{i+1/2} =
U^{n,+}_{i+1/2}= U^n_{i+1/2},$ for all $i$. Therefore by the
consistency of the numerical flux we have that   
\begin{align}
  \begin{split}
    \mcal{F}(U^{n,-}_{i+1/2}, U^{n,+}_{i+1/2}) -
    \mcal{F}(U^{n,-}_{i-1/2}, U^{n,+}_{i-1/2}) + S^{n,-}_{i+1/2} +
    S^{n,+}_{i-1/2} &= F(U^n_{i+1/2}) - F(U^n_{i-1/2}) +
    S^n_{i+1/2}+ S^n_{i-1/2}\\ 
    &= 0.
  \end{split}
\end{align}
Thus the hyperbolic flux term $\mcal{F}$ and the source term balance
each other for the stationary solution. Now for the mass equation, the
discretisation of the parabolic term $c_2 \partial_{xx}\left(\frac{q^2}{\rho}
\right)$ vanishes due to the consistency of the numerical fluxes because of the
fact that $q_i^n=0$ for all $i$. The discrete form of the hydrostatic
solution then ensures that the central discretisation of the parabolic term
$c_2\partial_{xx}P(\rho)$ and the potential term
$\partial_{x}(\rho\D_x\phi)$ balance each other in the hydrostatic
case, and therefore the scheme preserves the well-balancing property.  

Next, we prove (iii), the AP property which follows from the modified
mass equation   
\begin{align}
  \label{eq:stfd_mas1}
    \begin{split}
\rho^{n+1}_i&=\rho^n_i-c_1 \frac{\mcal{F}^\rho (U_{i+1/2}^{n,-},U_{i+1/2}^{n,+})-\mcal{F}^\rho(U_{i-1/2}^{n,-},U_{i-1/2}^{n,+})}{\Dlx} + c_1 \Dlt \frac{\left(\frac{{q_{i+1}^n}^2}{\rho_{i+1}^n}\right)-\left(\frac{{2{q_{i}^n}^2}}{\rho_{i}^n}\right)+\left(\frac{{q_{i-1}^n}^2}{\rho_{i-1}^n}\right)}{\Dlx ^2} 
\\[9pt] &+\frac{\veps^{1-\beta} \Dlt^2 }{\veps^{1+\beta} + \Dlt} \frac{P(\rho_{i+1}^n)-2P(\rho_{i}^n)+P(\rho_{i-1}^n)}{\Dlx^2} -  \frac{ \Dlt^2 }{\veps^{1+\beta} +\Dlt} \frac{\bar{\rho}^n_{i+1/2}( \phi_{i+1}^n-\phi_{i}^n)-\bar{\rho}^n_{i-1/2}( \phi_{i}^n-\phi_{i-1}^n)}{\Dlx ^2}.
    \end{split}
\end{align}
We set $\beta = 1$ in \eqref{eq:stfd_mas1} and take the limit $\veps
\to 0$. Noting that $c_1 \to 0$ we get  
\begin{align}
    \rho^{n+1}_i&=\rho^n_i + \Dlt \frac{P(\rho_{i+1}^n)-2P(\rho_{i}^n)+P(\rho_{i-1}^n)}{\Dlx^2} -  \Dlt \frac{\bar{\rho}^n_{i+1/2}( \phi_{i+1}^n-\phi_{i}^n)-\bar{\rho}^n_{i-1/2}( \phi_{i}^n-\phi_{i-1}^n)}{\Dlx ^2},
\end{align}
which is an explicit discretisation of the porous medium equation. In
an analogous manner, for $\beta \in [0,1),$ we obtain
\begin{align}
    \rho^{n+1}_i&=\rho^n_i -  \Dlt \frac{\bar{\rho}^n_{i+1/2}( \phi_{i+1}^n-\phi_{i}^n)-\bar{\rho}^n_{i-1/2}( \phi_{i}^n-\phi_{i-1}^n)}{\Dlx ^2}
\end{align}
which is a consistent discretisation of the transport equation.

Thus, the fully discrete scheme
\eqref{eq:stfd_mas}-\eqref{eq:stfd_mom} is well-balanced that relaxes
to a consistent discretisation of the asymptotic limit independent of
$\beta$.   
\end{proof}

\section{Numerical Case Studies}
\label{sec:numer-case-stud}

In this section, we test the proposed scheme in order to study its
unified AP and well-balancing properties. We compare the scheme with
both non well-balanced and non AP schemes, and verify that our scheme
performs better than them in the stiff as well as the non-stiff regimes. A
Rusanov-type approximate Riemann solver was used for the hyperbolic
numerical flux function $\mcal{F}$. A CFL condition of the form 
\begin{equation}
  \label{eq:CFL_cond}
  \Dlt=\lambda_\CFL 
  \min\left(\frac{\Dlx^2}{\veps^{1-\beta}},\frac{\Dlx}{|\partial_x \phi|}
  \right) 
\end{equation}
is used to compute the time step. The exact value of $\lambda_\CFL$
will be specified in each problem. 

\subsection{Unified Asymptotic Preserving Property}
\label{sec:Sod}
The goal of this test problem is to demonstrate the ability of the new
scheme to compute the flow characteristics for a wide range of
$\veps$. We consider two extreme cases, namely $\veps=1$ and
$\veps=0.001$ where the latter is to showcase the AP property in the
limit $\veps\to0$. In order to understand the multiscale behaviour of
the solver, we consider a simple Riemann problem, and therein we use
the Sod initial data under a gravitational field with potential
$\phi(x) = x$ as given in \cite{VC19}. We use extrapolation boundary
conditions and the computational domain is set to be $[0,1]$. The
initial data read  
\begin{equation}
  \label{eq:SodIC}
  (\rho, u) = 
  \begin{cases}
    (1,0), & \text{if}\ x < 0.5, \\
    (0.125,0), & \text{if}\ x > 0.5.
  \end{cases}
\end{equation}
The simulation is run until a final time of $T = 0.2$. Results are
presented in Figures~\ref{fig:sodeps1} and \ref{fig:sodeps10-3} for
the non-stiff regime ($\veps = 1$) and the stiff regime ($\veps =
0.001$) respectively. To test the convergence of the scheme, for
$\veps = 1$, we compare the numerical solution obtained on a coarse
mesh ($N=100$) with that on a fine mesh ($N=1000$) for
$\lambda_\CFL=0.45$. It can be seen that the solution computed on a
coarse mesh is in good agreement with that on a fine mesh. The test
results contain a shock moving to the right followed by an expansion,
which shows the  efficacy of the scheme in resolving the fully
compressible flow features.     
\begin{figure}[htbp]
\centering
\begin{subfigure}{.4\textwidth}
  \centering
  \includegraphics[width=\linewidth]{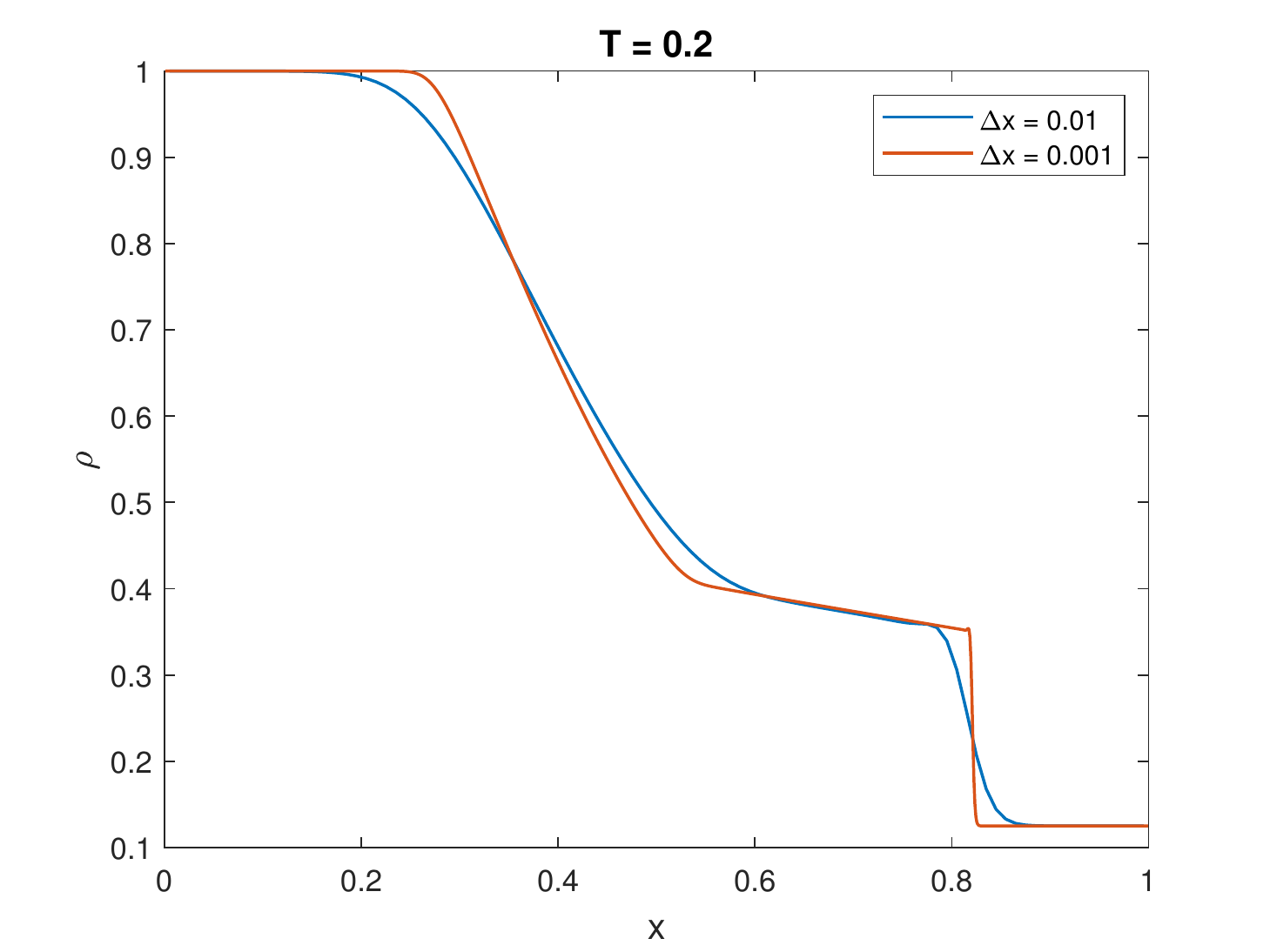}
  \caption{}
\end{subfigure}%
\begin{subfigure}{.4\textwidth}
  \centering
  \includegraphics[width=\linewidth]{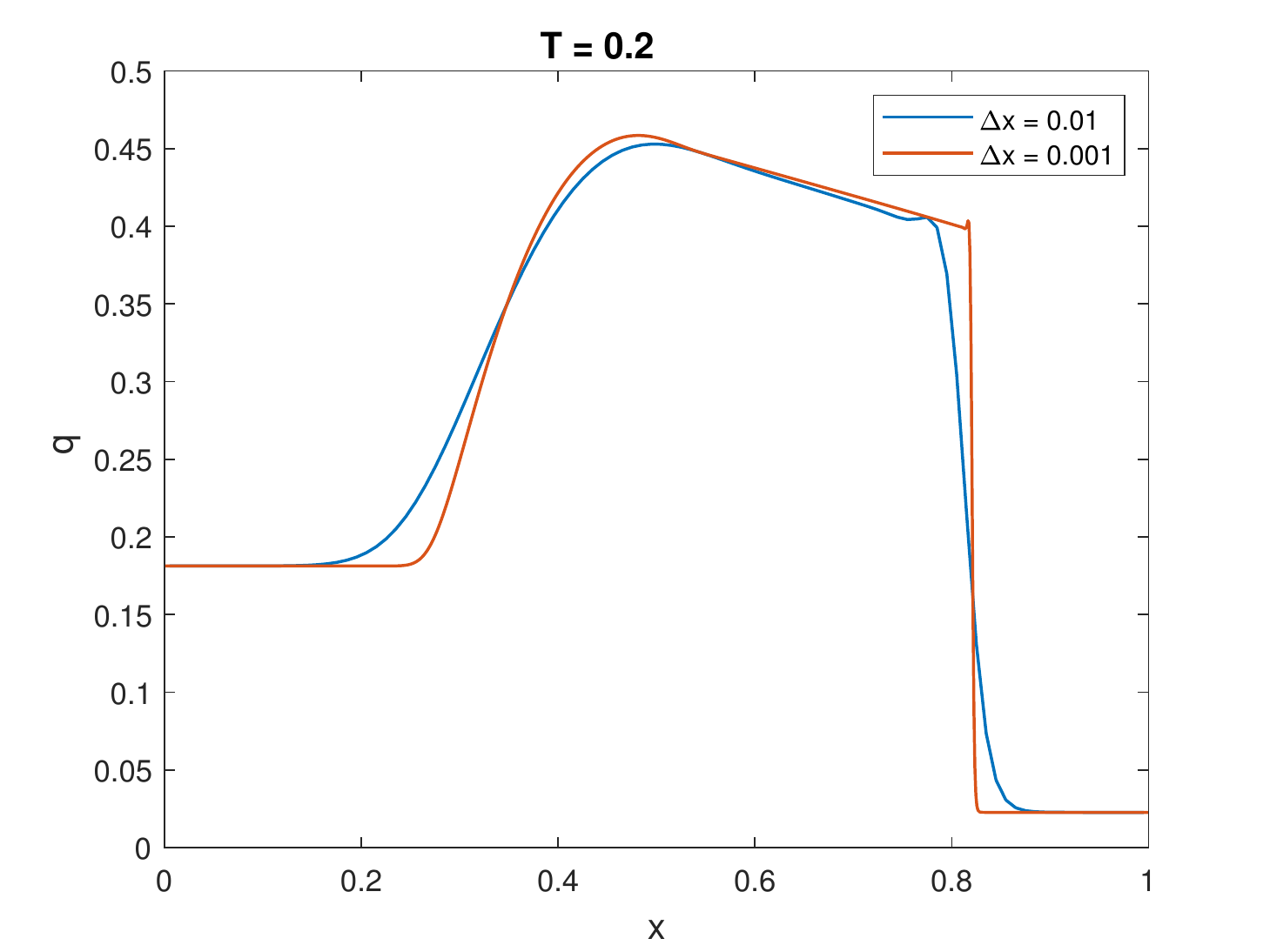}
  \caption{}
\end{subfigure}
\caption{1D Sod problem: solution profiles of (A) $\rho$ and (B) $q$
  at $T=0.2$ in the non-stiff regime for $\veps=1$.} 
\label{fig:sodeps1}
\end{figure}
\begin{figure}
\centering
\begin{subfigure}{.4\textwidth}
  \centering
  \includegraphics[width=\linewidth]{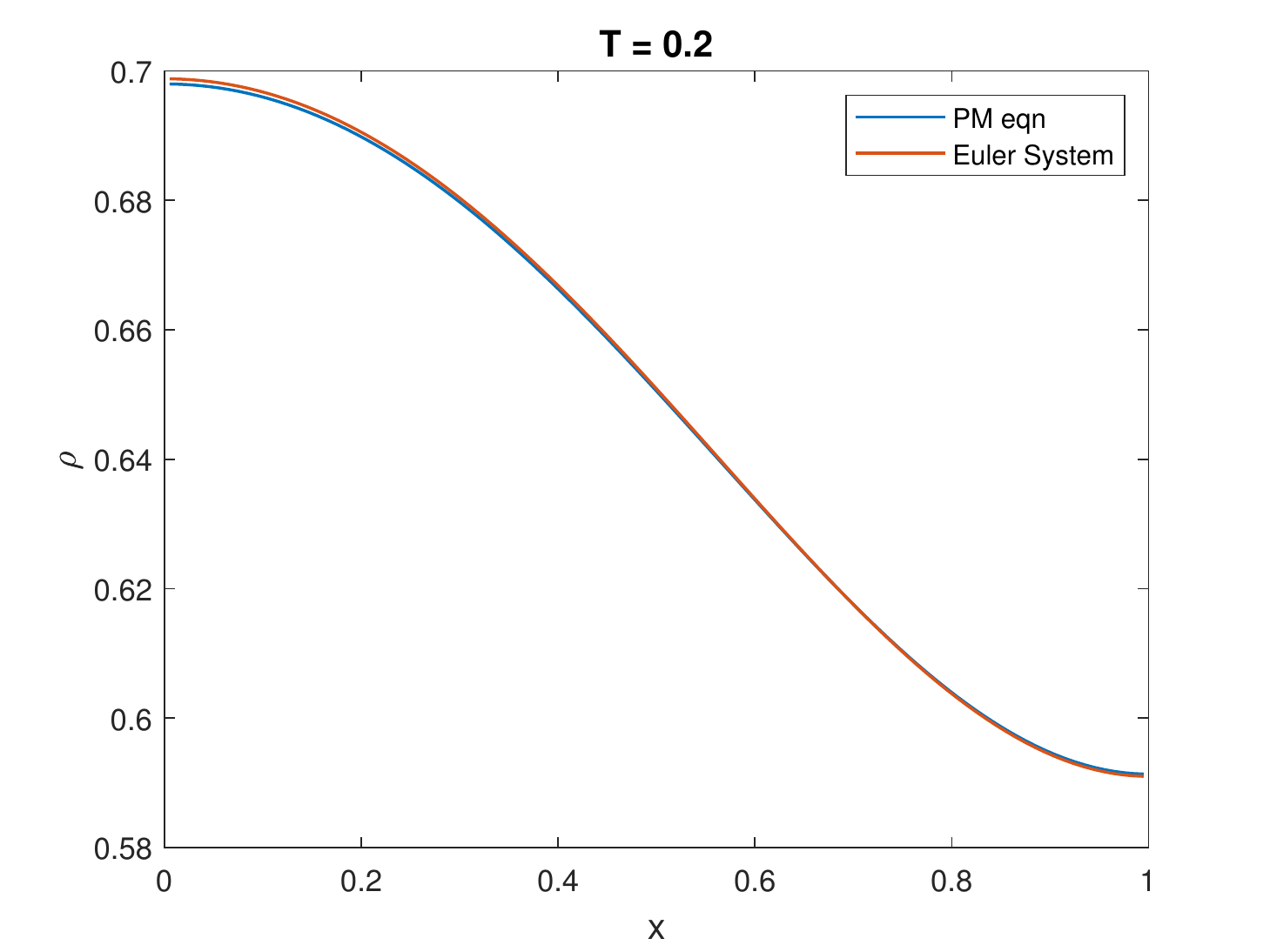}
  \caption{}
\end{subfigure}%
\begin{subfigure}{.4\textwidth}
  \centering
  \includegraphics[width=\linewidth]{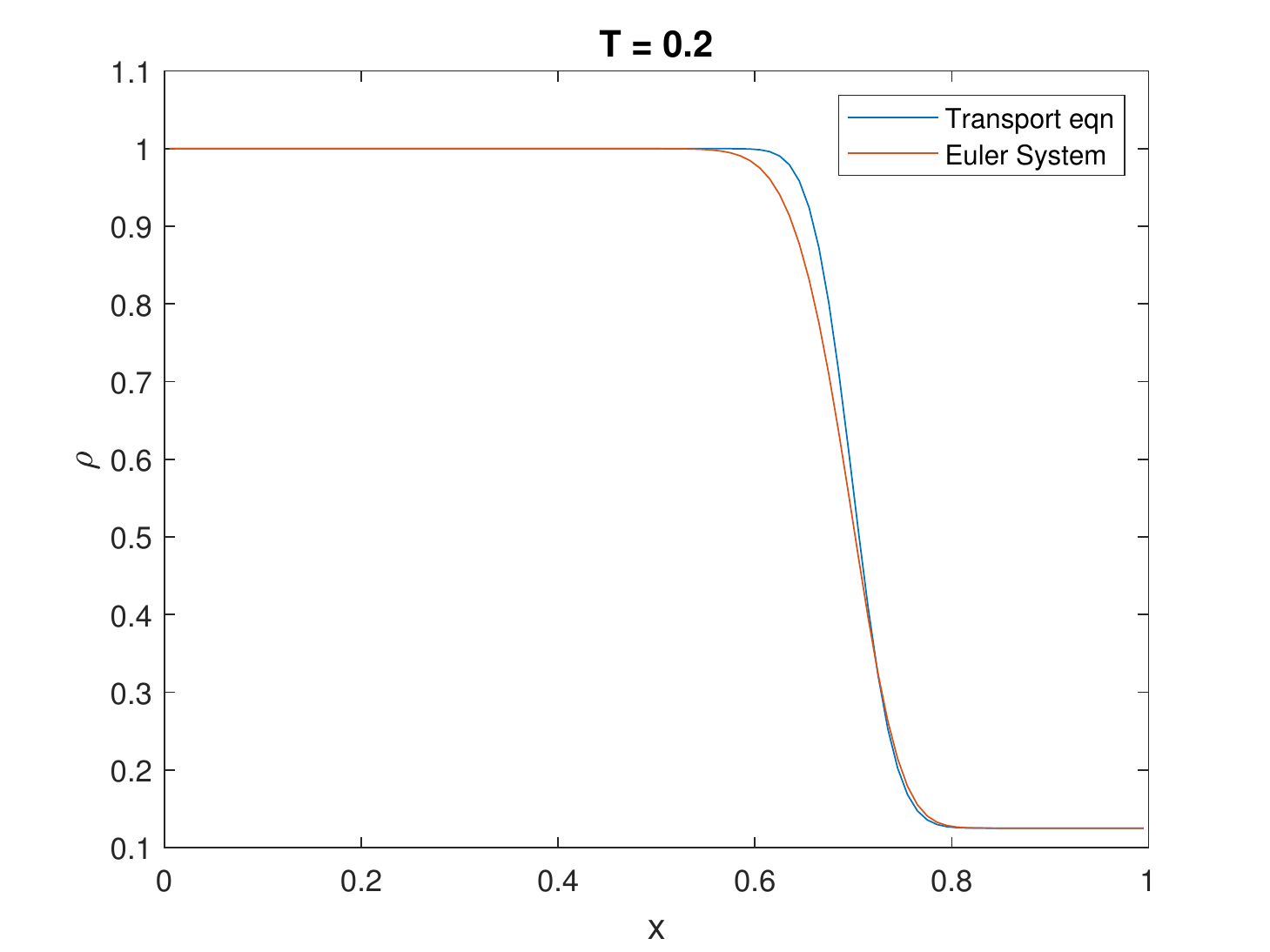}
  \caption{}
\end{subfigure}
\caption{1D Sod problem: solution profiles of $\rho$ at $T=0.2$ for
  $\veps=0.001$ in the (A) parabolic and (B) hyperbolic relaxation
  regimes.}  
\label{fig:sodeps10-3}
\end{figure}

In the stiff regime, we consider both the hyperbolic and parabolic
relaxations. When $\veps = 0.001$ by setting $\beta = 1$, the
numerical solution obtained on a mesh with $100$ cells with
$\lambda_\CFL=0.45$, is compared with that of a standard first order
scheme for the porous medium equation \eqref{eq:pmeqn} to demonstrate
the parabolic relaxation. The resulting solution is in perfect agreement
with that of the parabolic equation showing the correct asymptotic
behaviour of the scheme. Similarly, to show the hyperbolic relaxation
for $\beta \in [0,1)$, the scheme is tested for $\beta = 0.1$ and
$\lambda_\CFL=0.45$. The result shows good agreement with a first
order upwind scheme applied to the transport equation
\eqref{eq:transporteqn}. 

\subsection{Well-balancing Property}
\label{sec:wb}

To numerically validate the well-balancing property of the scheme, we use
initial data in both isothermal and isentropic hydrostatic
equilibrium, taken from \cite{VC19}. We also add a small perturbation
to these equilibria to study their evolution via the scheme. Finally,
we compute the solution for a large time to show the convergence of
the numerical solution to the steady state, and to compare our results
with that of a non well-balanced solver. The CFL number
$\lambda_\CFL$ was taken to be 0.45 in every case.   

\subsubsection{Isothermal Hydrostatic Solution}
\label{sec:isotherm}

We solve the system which is initially in isothermal hydrostatic
equilibrium for $\beta = 1$ corresponding to the following
configuration: 
\begin{equation}
  \label{eq:isotherm}
  \rho(0,x) \equiv \rho_e(x) = \exp(\phi(x)), \quad u(0,x) \equiv
  u_e(x) = 0. 
\end{equation}
The exact solution $(\rho_e,u_e)$, cf.\ \eqref{eq:hydro_iso}, is
interpolated onto the grid, and for different values of $\veps$, the
$L^1$ errors are calculated for three different gravitational
potentials $\phi(x) = x, \frac{x^2}{2}, \sin(2 \pi x)$ for grids with
$100$ and $1000$ cells upto a final time $T = 2$. It can be seen from
Table \ref{Tab:isotherm_err} that the scheme exhibits good precision in
approximating the exact hydrostatic solution for the different
potentials in both the stiff and non-stiff regimes. Thus we conclude
that the scheme maintains the well-balancing property for isothermal
hydrostatic solutions.

\begin{table}[h]
\begin{subtable}{.51\linewidth}
\small
\begin{tabular}{cccc}
\hline
\multicolumn{4}{|c|}{$\veps = 1.0$}                                                                                                                    \\ \hline
\multicolumn{1}{|c|}{$\phi$}                          & \multicolumn{1}{c|}{$N$}    & \multicolumn{1}{c|}{Error in $\rho$} & \multicolumn{1}{c|}{Error in $q$} \\ \hline
\multicolumn{1}{|c|}{\multirow{2}{*}{$x$}}            & \multicolumn{1}{c|}{100}  & \multicolumn{1}{c|}{2.6815E-06}      & \multicolumn{1}{c|}{4.1644E-06}   \\ \cline{2-4} 
\multicolumn{1}{|c|}{}                                & \multicolumn{1}{c|}{1000} & \multicolumn{1}{c|}{2.7018E-08}      & \multicolumn{1}{c|}{4.1655E-08}   \\ \hline
\multicolumn{1}{|c|}{\multirow{2}{*}{$x^2/2$}}        & \multicolumn{1}{c|}{100}  & \multicolumn{1}{c|}{9.9540E-07}      & \multicolumn{1}{c|}{8.6076E-07}   \\ \cline{2-4} 
\multicolumn{1}{|c|}{}                                & \multicolumn{1}{c|}{1000} & \multicolumn{1}{c|}{9.7002E-09}      & \multicolumn{1}{c|}{8.4871E-09}   \\ \hline
\multicolumn{1}{|c|}{\multirow{2}{*}{$\sin(2 \pi x)$}} & \multicolumn{1}{c|}{100}  & \multicolumn{1}{c|}{2.1751E-04}      & \multicolumn{1}{c|}{4.3318E-06}   \\ \cline{2-4} 
\multicolumn{1}{|c|}{}                                & \multicolumn{1}{c|}{1000} & \multicolumn{1}{c|}{2.0620E-06}      & \multicolumn{1}{c|}{6.2102E-08}   \\ \hline
\end{tabular}
\end{subtable}
\begin{subtable}{.48\linewidth}
\small
\begin{tabular}{llll}
\hline
\multicolumn{4}{|c|}{$\veps = 0.1$}                                                                                                                                   \\ \hline
\multicolumn{1}{|c|}{$\phi$}                          & \multicolumn{1}{c|}{$N$}    & \multicolumn{1}{c|}{Error in $\rho$} & \multicolumn{1}{c|}{Error in $q$} \\ \hline
\multicolumn{1}{|c|}{\multirow{2}{*}{$x$}}            & \multicolumn{1}{c|}{100}  & \multicolumn{1}{c|}{1.4173E-06}      & \multicolumn{1}{c|}{1.1110E-05}   \\ \cline{2-4} 
\multicolumn{1}{|c|}{}                                & \multicolumn{1}{c|}{1000} & \multicolumn{1}{c|}{1.4070E-08}      & \multicolumn{1}{c|}{1.0878E-07}   \\ \hline
\multicolumn{1}{|c|}{\multirow{2}{*}{$x^2/2$}}        & \multicolumn{1}{c|}{100}  & \multicolumn{1}{c|}{8.9859E-07}      & \multicolumn{1}{c|}{2.1290E-06}   \\ \cline{2-4} 
\multicolumn{1}{|c|}{}                                & \multicolumn{1}{c|}{1000} & \multicolumn{1}{c|}{8.7053E-09}      & \multicolumn{1}{c|}{2.0610E-08}   \\ \hline
\multicolumn{1}{|c|}{\multirow{2}{*}{$\sin(2 \pi x)$}} & \multicolumn{1}{c|}{100}  & \multicolumn{1}{c|}{2.2498E-04}      & \multicolumn{1}{c|}{1.4161E-05}   \\ \cline{2-4} 
\multicolumn{1}{|c|}{}                                & \multicolumn{1}{c|}{1000} & \multicolumn{1}{c|}{2.1303E-06}      & \multicolumn{1}{c|}{1.4130E-08}   \\ \hline
\end{tabular}
\end{subtable}%
\vspace{5mm}

\begin{subtable}{.51\linewidth}
\small
\begin{tabular}{llll}
\hline
\multicolumn{4}{|c|}{$\veps = 0.01$}                                                                                                                  \\ \hline
\multicolumn{1}{|c|}{$\phi$}                          & \multicolumn{1}{c|}{$N$}    & \multicolumn{1}{c|}{Error in $\rho$} & \multicolumn{1}{c|}{Error in $q$} \\ \hline
\multicolumn{1}{|c|}{\multirow{2}{*}{$x$}}            & \multicolumn{1}{c|}{100}  & \multicolumn{1}{c|}{1.1902E-06}      & \multicolumn{1}{c|}{1.3049E-05}   \\ \cline{2-4} 
\multicolumn{1}{|c|}{}                                & \multicolumn{1}{c|}{1000} & \multicolumn{1}{c|}{1.1544E-08}      & \multicolumn{1}{c|}{1.2975E-07}   \\ \hline
\multicolumn{1}{|c|}{\multirow{2}{*}{$x^2/2$}}        & \multicolumn{1}{c|}{100}  & \multicolumn{1}{c|}{8.8290E-07}      & \multicolumn{1}{c|}{2.4434E-06}   \\ \cline{2-4} 
\multicolumn{1}{|c|}{}                                & \multicolumn{1}{c|}{1000} & \multicolumn{1}{c|}{8.5468E-09}      & \multicolumn{1}{c|}{2.4053E-08}   \\ \hline
\multicolumn{1}{|c|}{\multirow{2}{*}{$\sin(2 \pi x)$}} & \multicolumn{1}{c|}{100}  & \multicolumn{1}{c|}{2.2599E-04}      & \multicolumn{1}{c|}{1.6097E-05}   \\ \cline{2-4} 
\multicolumn{1}{|c|}{}                                & \multicolumn{1}{c|}{1000} & \multicolumn{1}{c|}{2.1313E-06}      & \multicolumn{1}{c|}{1.6134E-08}   \\ \hline
\end{tabular}
\end{subtable}
\begin{subtable}{.48\linewidth}
\small
\begin{tabular}{llll}
\hline
\multicolumn{4}{|c|}{$\veps = 0.001$}                                                                                                                   \\ \hline
\multicolumn{1}{|c|}{$\phi$}                          & \multicolumn{1}{c|}{$N$}    & \multicolumn{1}{c|}{Error in $\rho$} & \multicolumn{1}{c|}{Error in $q$} \\ \hline
\multicolumn{1}{|c|}{\multirow{2}{*}{$x$}}            & \multicolumn{1}{c|}{100}  & \multicolumn{1}{c|}{1.1705E-06}      & \multicolumn{1}{c|}{1.3170E-05}   \\ \cline{2-4} 
\multicolumn{1}{|c|}{}                                & \multicolumn{1}{c|}{1000} & \multicolumn{1}{c|}{1.1390E-08}      & \multicolumn{1}{c|}{1.3172E-07}   \\ \hline
\multicolumn{1}{|c|}{\multirow{2}{*}{$x^2/2$}}        & \multicolumn{1}{c|}{100}  & \multicolumn{1}{c|}{8.8199E-07}      & \multicolumn{1}{c|}{2.4605E-06}   \\ \cline{2-4} 
\multicolumn{1}{|c|}{}                                & \multicolumn{1}{c|}{1000} & \multicolumn{1}{c|}{8.5282E-09}      & \multicolumn{1}{c|}{2.4367E-08}   \\ \hline
\multicolumn{1}{|c|}{\multirow{2}{*}{$\sin(2 \pi x)$}} & \multicolumn{1}{c|}{100}  & \multicolumn{1}{c|}{2.2606E-04}      & \multicolumn{1}{c|}{1.6197E-05}   \\ \cline{2-4} 
\multicolumn{1}{|c|}{}                                & \multicolumn{1}{c|}{1000} & \multicolumn{1}{c|}{2.1314E-06}      & \multicolumn{1}{c|}{1.6336E-08}   \\ \hline
\end{tabular}
\end{subtable}
\caption{Errors in the density $\rho$ and the momentum $q$ for
  different potentials and for a range of $\veps$ using
  different mesh sizes in the isothermal test case.} 
\label{Tab:isotherm_err}
\end{table}

Next, we want to study the efficacy of the scheme in simulating the
evolution of small perturbations added to the initial equilibrium
solution. To this end, we compare our solver with a non well-balanced
scheme which also makes use of the unified AP time discretisation but
without the equilibrium spatial reconstruction. The potential in this
case is taken to be $\phi(x) = x$ and the initial density is given by 
\begin{equation}
  \label{eq:isothermpert}
  \rho_0(x) = \exp(\phi(x)) + \zeta \exp(-100(x-0.5)^2).
\end{equation}
The computational domain is $[0,1]$ and the boundary conditions are
imposed by interpolating the equilibrium solution onto the ghost
cells. The results are presented for two different amplitudes of
perturbation $\zeta = 10^{-3}, 10^{-5}$. In the non-stiff regime, as
expected, the well-balanced scheme resolves the solution well when
compared to the non well-balanced scheme, cf.\
Figure~\ref{fig:it_pert_eps1}. However, we notice from Figure
\ref{fig:it_pert_eps10-3} that the non well-balanced scheme also
produces results similar to that of the well-balanced scheme in the
stiff regime. This behaviour can be explained by the fact that the
parabolic porous medium equation \eqref{eq:pmeqn} and the hyperbolic
Euler system \eqref{eq:euler_nd_mas}-\eqref{eq:euler_nd_mom} share the
same stationary state, and the AP property of the non well-balanced
scheme then ensures that it also relaxes to a reasonable approximation of
the same steady state as $\veps \to 0$. 
\begin{figure}
\centering
\begin{subfigure}{.32\textwidth}
  \centering
  \includegraphics[width=\linewidth]{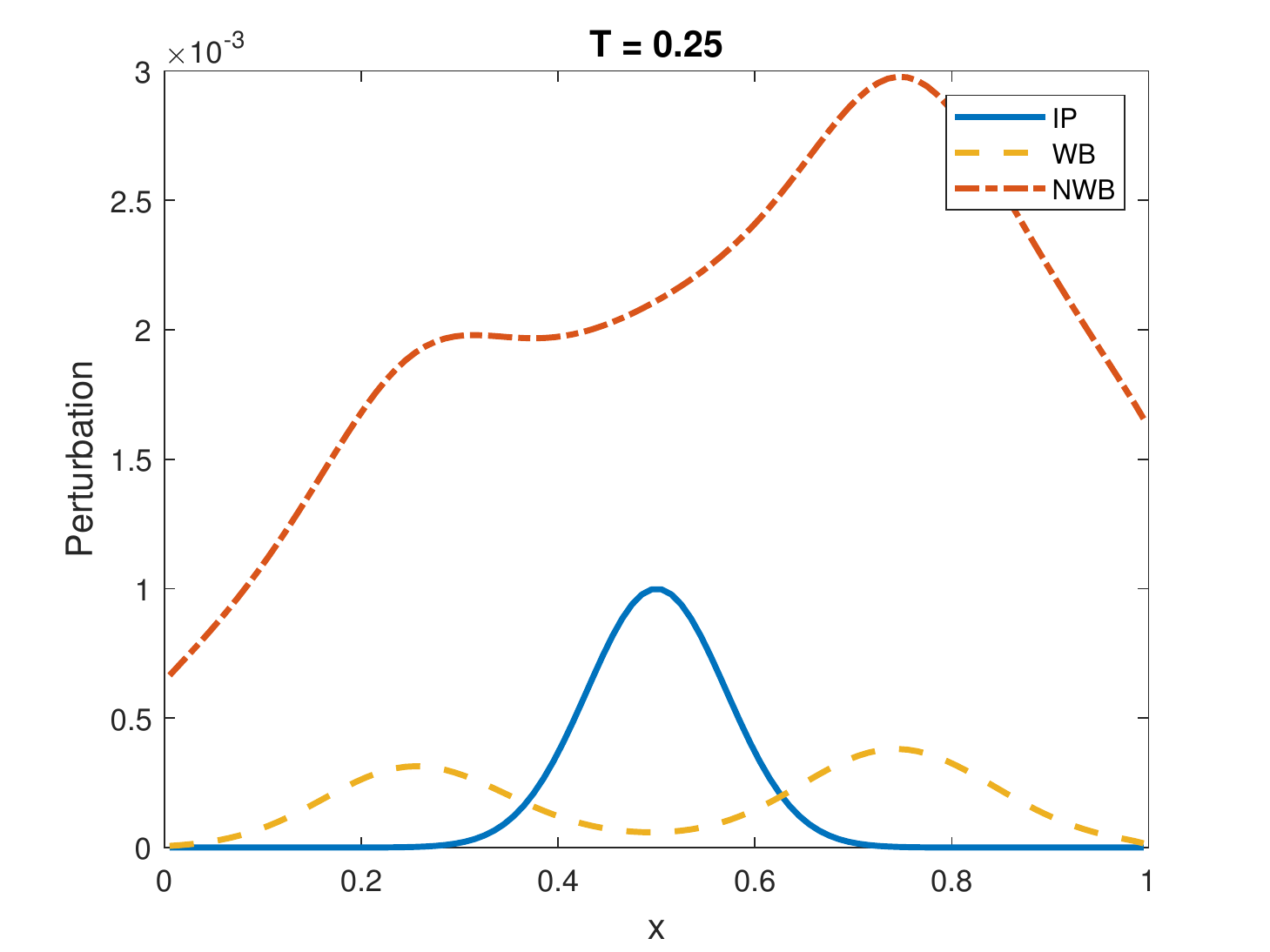}
  \caption{}
\end{subfigure}%
\begin{subfigure}{.32\textwidth}
  \centering
  \includegraphics[width=\linewidth]{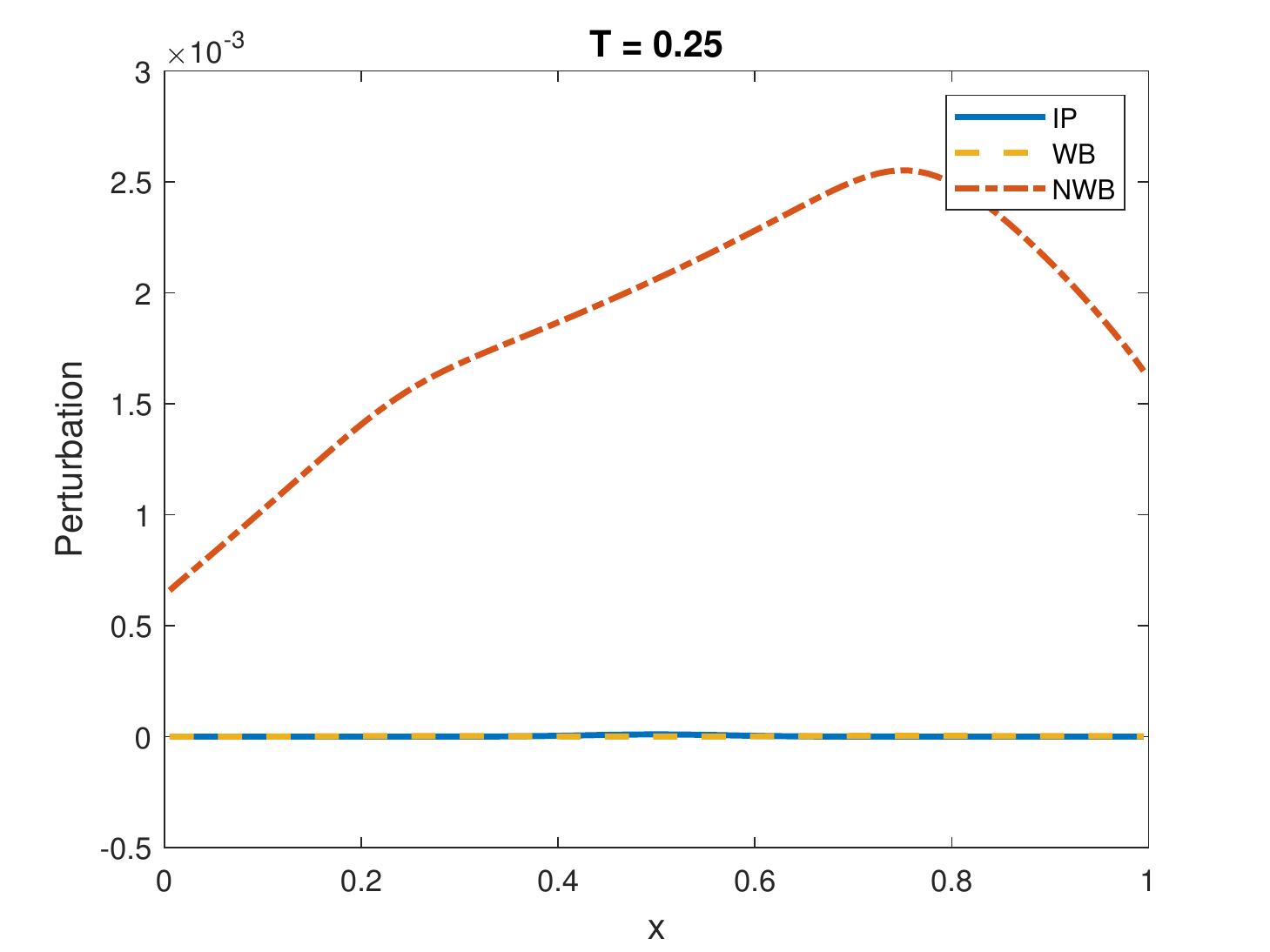}
  \caption{}
\end{subfigure}
\begin{subfigure}{.32\textwidth}
  \centering
  \includegraphics[width=\linewidth]{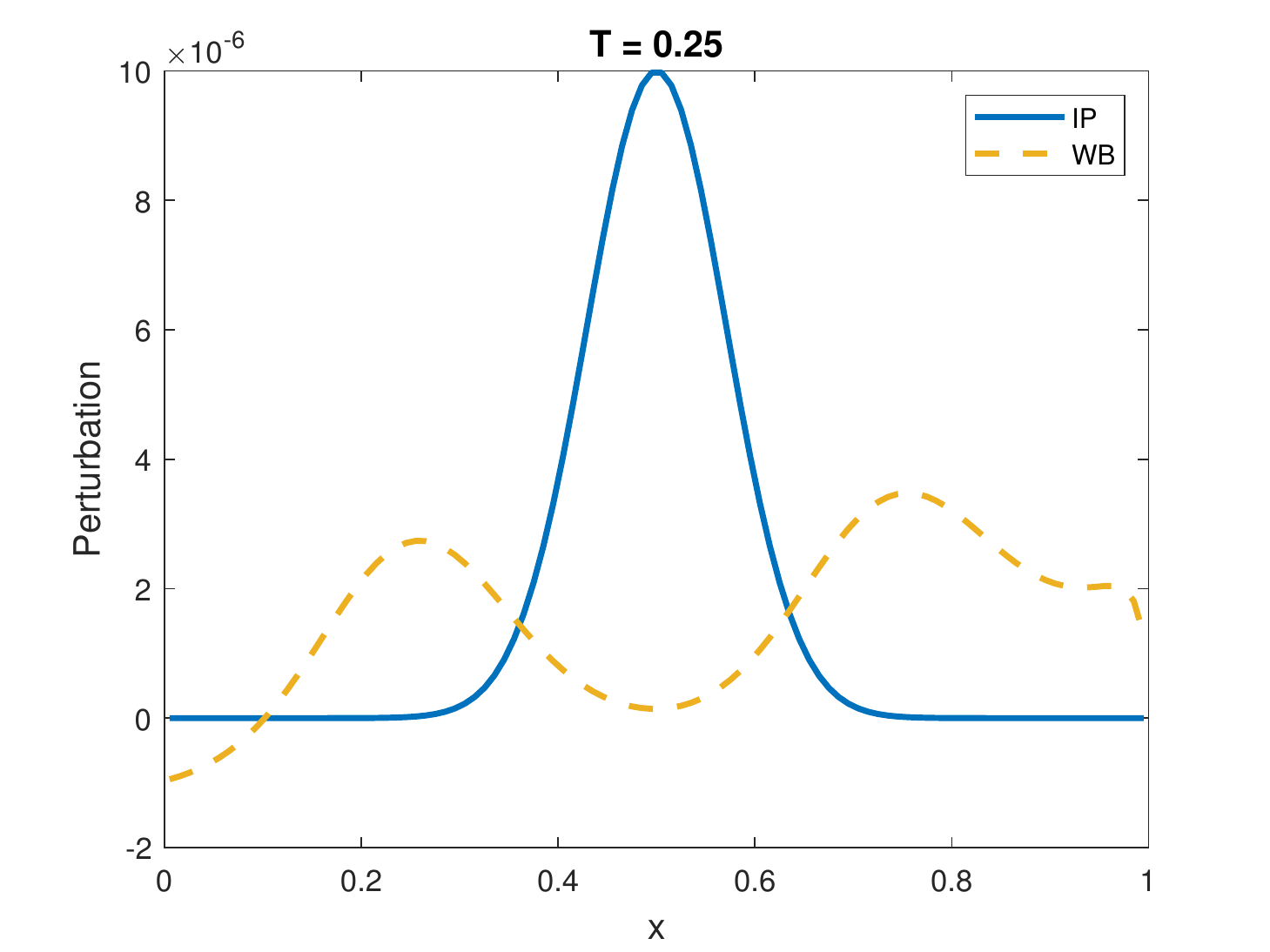}
  \caption{}
\end{subfigure}
\caption{Isothermal test: evolution in the density perturbation in the
  non-stiff regime for $\veps=1$. Comparison of well-balanced
  scheme with non-well balanced scheme for different values of
  $\zeta$. (A) $\zeta = 10^{-3}$ (B) $\zeta = 10^{-5}$. (C) Comparison
  of well-balanced scheme with initial perturbation for $\zeta =
  10^{-5}$.}  
\label{fig:it_pert_eps1}
\end{figure}
\begin{figure}
\centering
\begin{subfigure}{.4\textwidth}
  \centering
  \includegraphics[width=\linewidth]{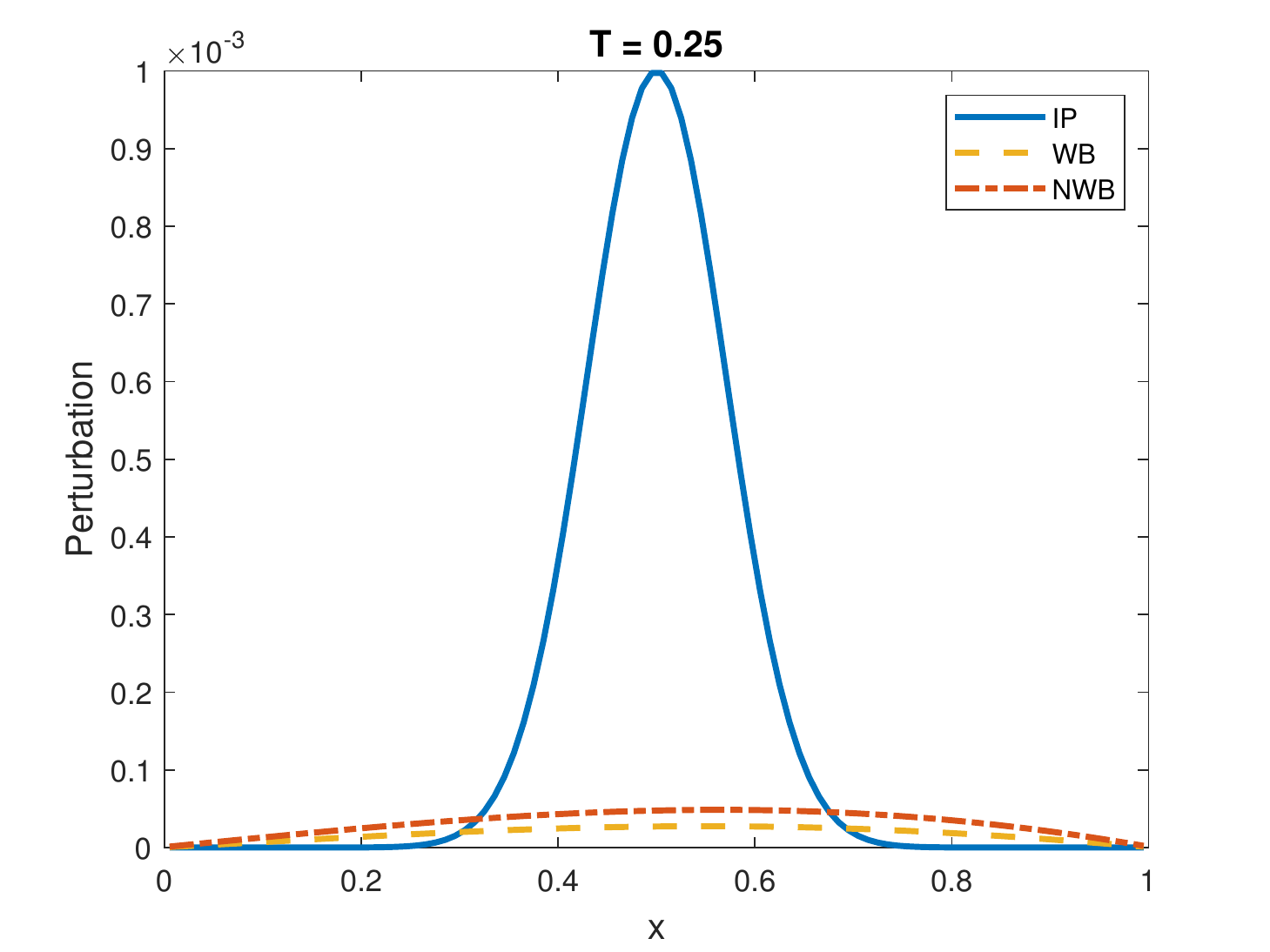}
  \caption{}
\end{subfigure}%
\begin{subfigure}{.4\textwidth}
  \centering
  \includegraphics[width=\linewidth]{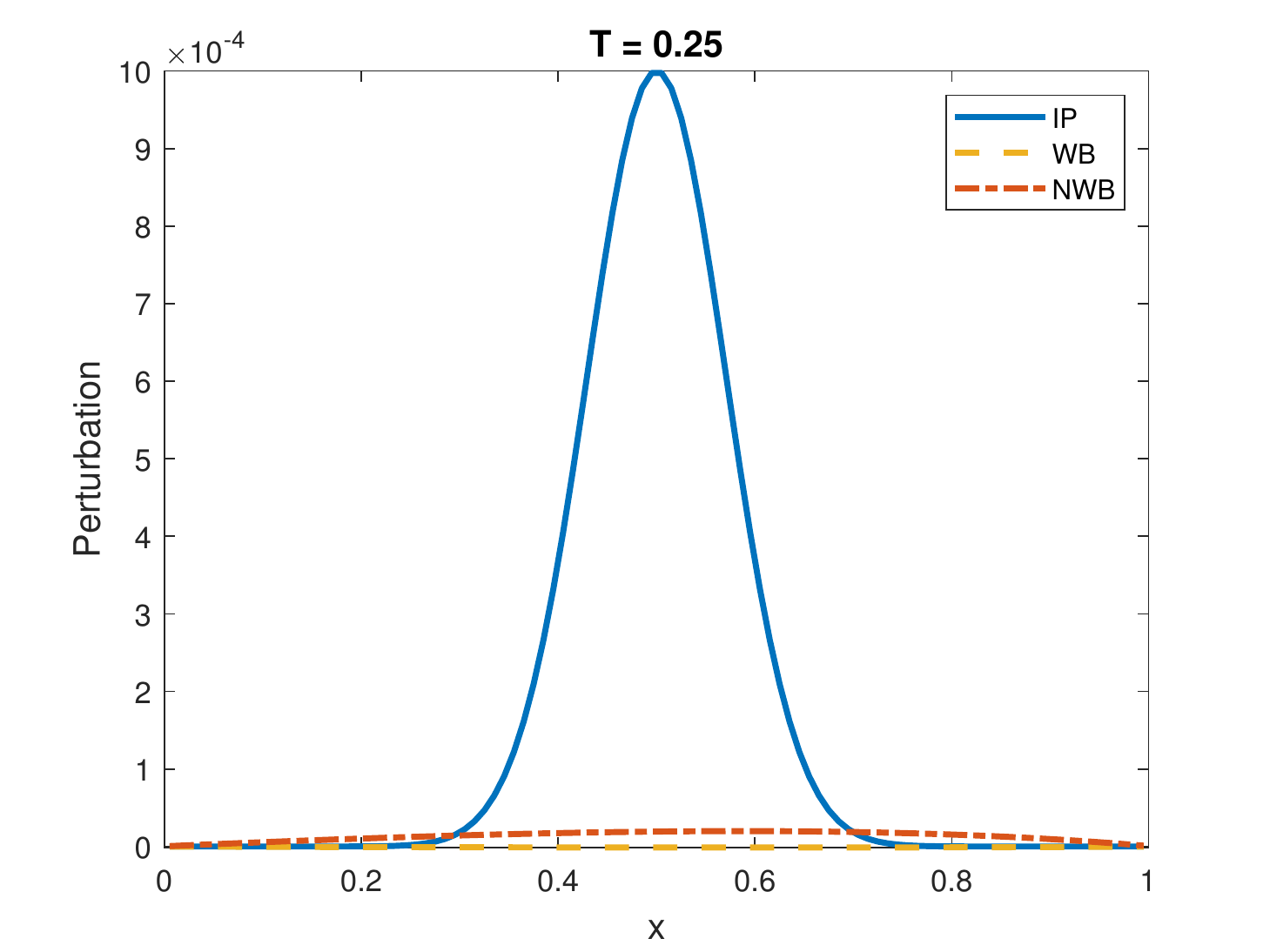}
  \caption{}
\end{subfigure}
\caption{Isothermal test: evolution in the density perturbation in the stiff regime for
  $\veps=0.001$ for (a) $\zeta = 10^{-3}$ (b) $\zeta = 10^{-5}$.}  
\label{fig:it_pert_eps10-3}
\end{figure}

Finally we also study the long time behaviour of the scheme to
corroborate its convergence to a steady state by simulating the 
isothermal hydrostatic solution until a large time $T=100$. We take
the initial data \eqref{eq:isothermpert} at equilibrium with a small
perturbation with amplitude $\zeta=10^{-3}$. The boundary conditions
again the exact solution interpolated onto the grid for the domain
$[0,1]$. For $\veps = 1$, the momentum converges accurately for the
well-balanced scheme but the non well-balanced scheme diverges away
from the stationary solution, cf.\ Figure~\ref{fig:lt_eps1}. 
\begin{figure}
\centering
\begin{subfigure}{.4\textwidth}
  \centering
  \includegraphics[width=\linewidth]{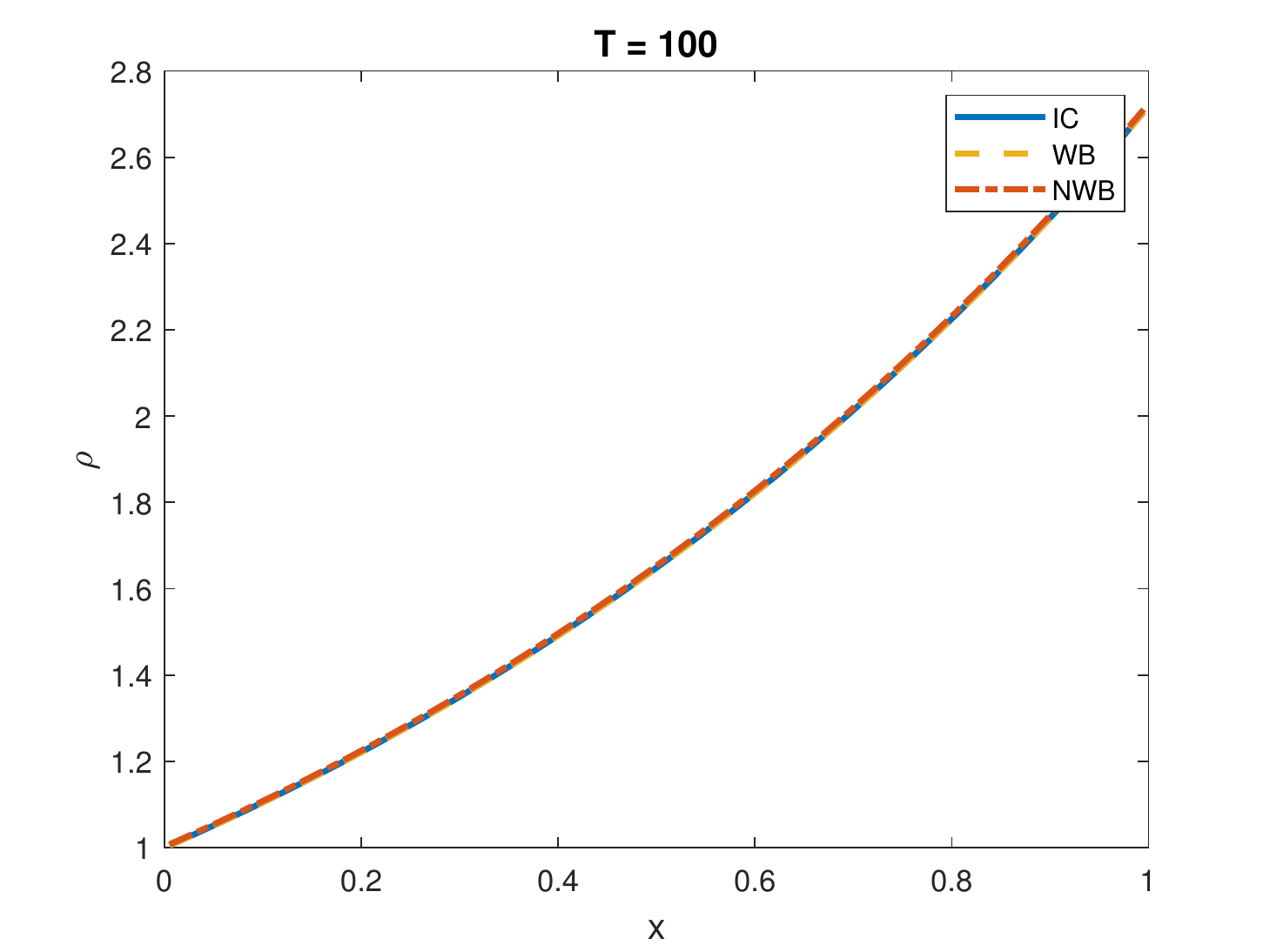}
  \caption{}
\end{subfigure}%
\begin{subfigure}{.4\textwidth}
  \centering
  \includegraphics[width=\linewidth]{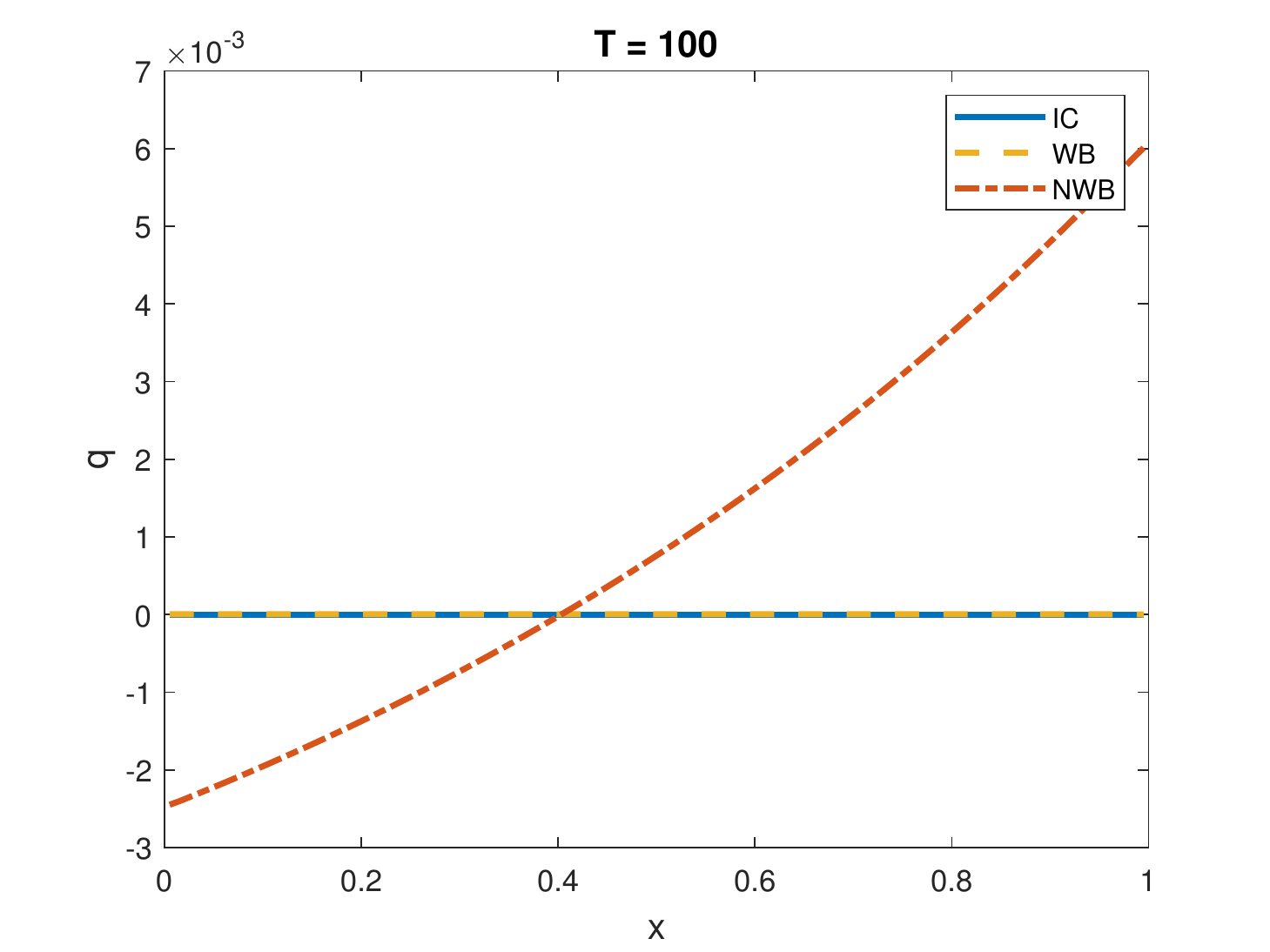}
  \caption{}
\end{subfigure}
\caption{Isothermal test: long time solution profiles when $\veps = 1$
  for (A) $\rho$ and (B) $q$.}   
\label{fig:lt_eps1}
\end{figure}
\begin{figure}
\centering
\begin{subfigure}{.4\textwidth}
  \centering
  \includegraphics[width=\linewidth]{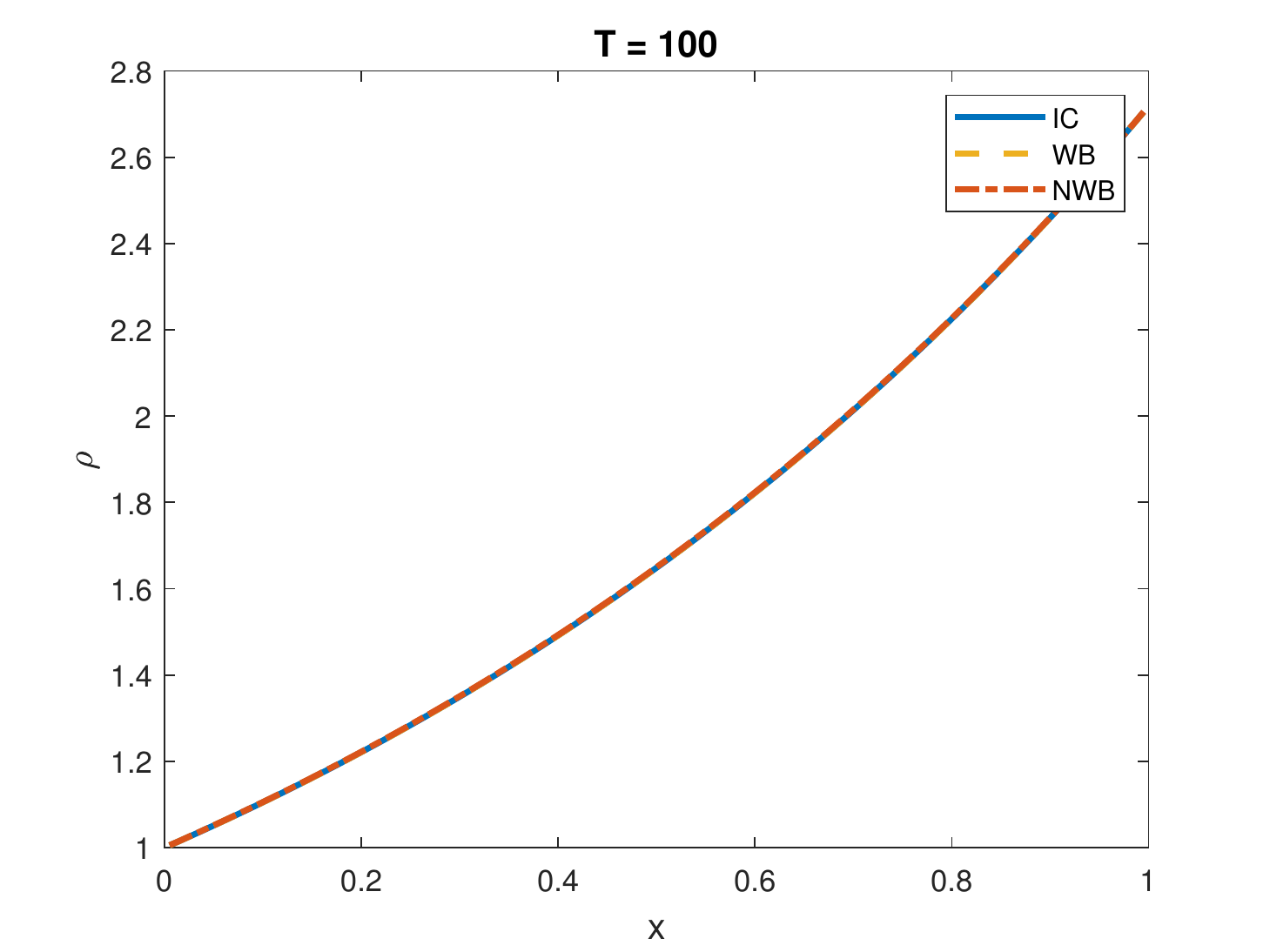}
  \caption{}
\end{subfigure}%
\begin{subfigure}{.4\textwidth}
  \centering
  \includegraphics[width=\linewidth]{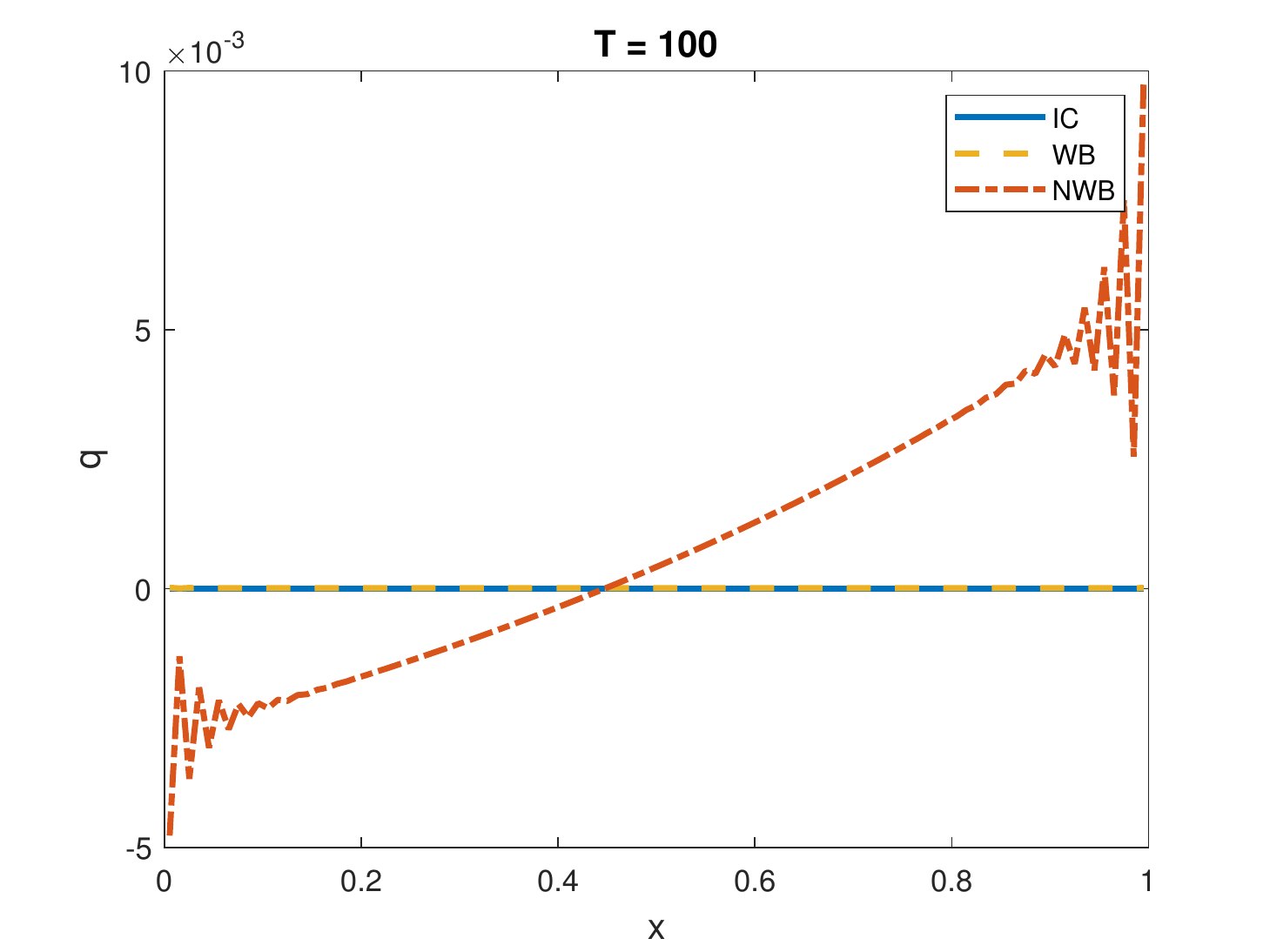}
  \caption{}
\end{subfigure}
\caption{Isothermal test: long time solution profiles when $\veps =
  0.1$ for (A) $\rho$ and (B) $q$.}  
\label{fig:lt_eps10-1}
\end{figure}
\begin{figure}
\centering
\begin{subfigure}{.4\textwidth}
  \centering
  \includegraphics[width=\linewidth]{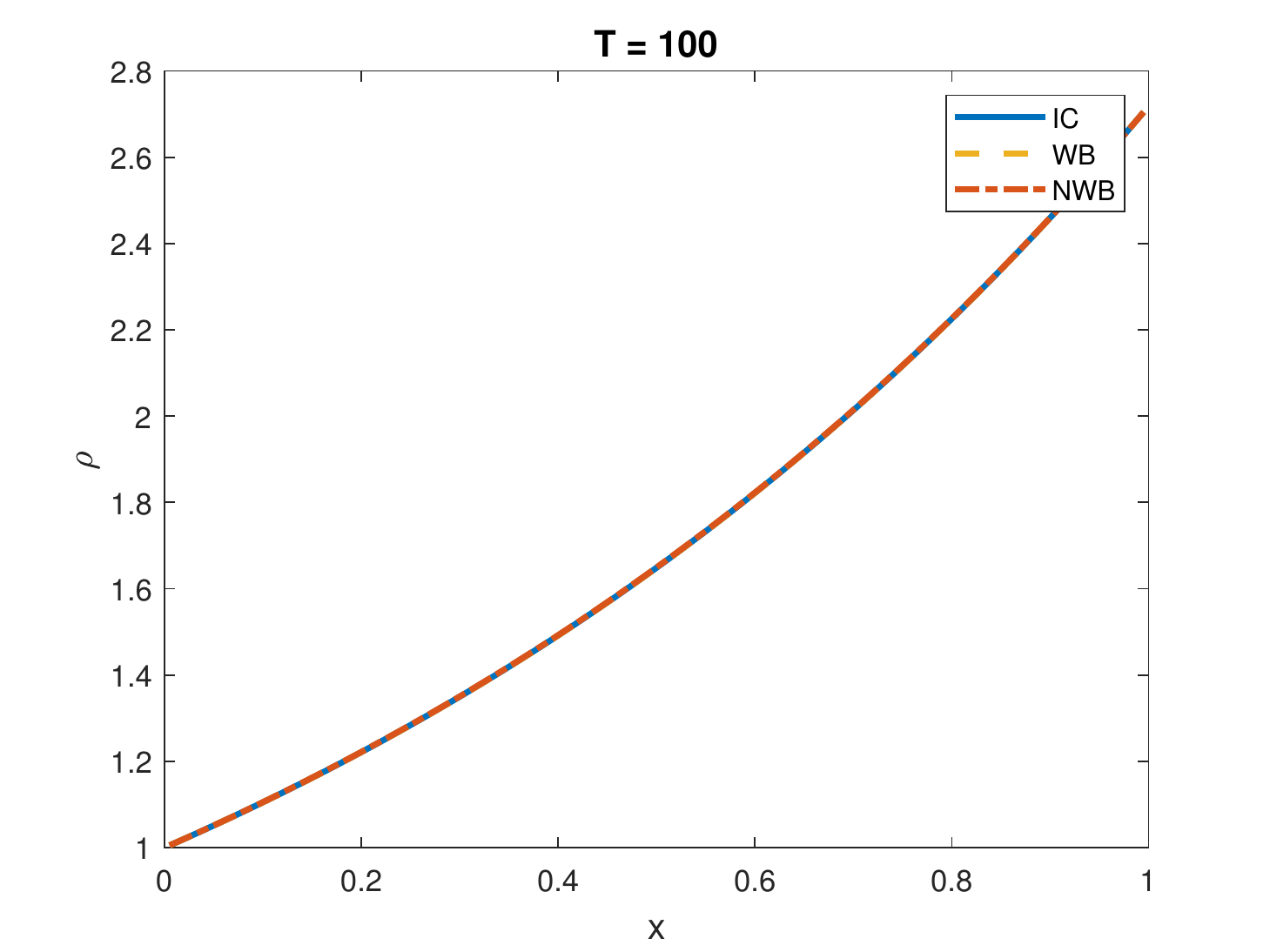}
  \caption{}
\end{subfigure}%
\begin{subfigure}{.4\textwidth}
  \centering
  \includegraphics[width=\linewidth]{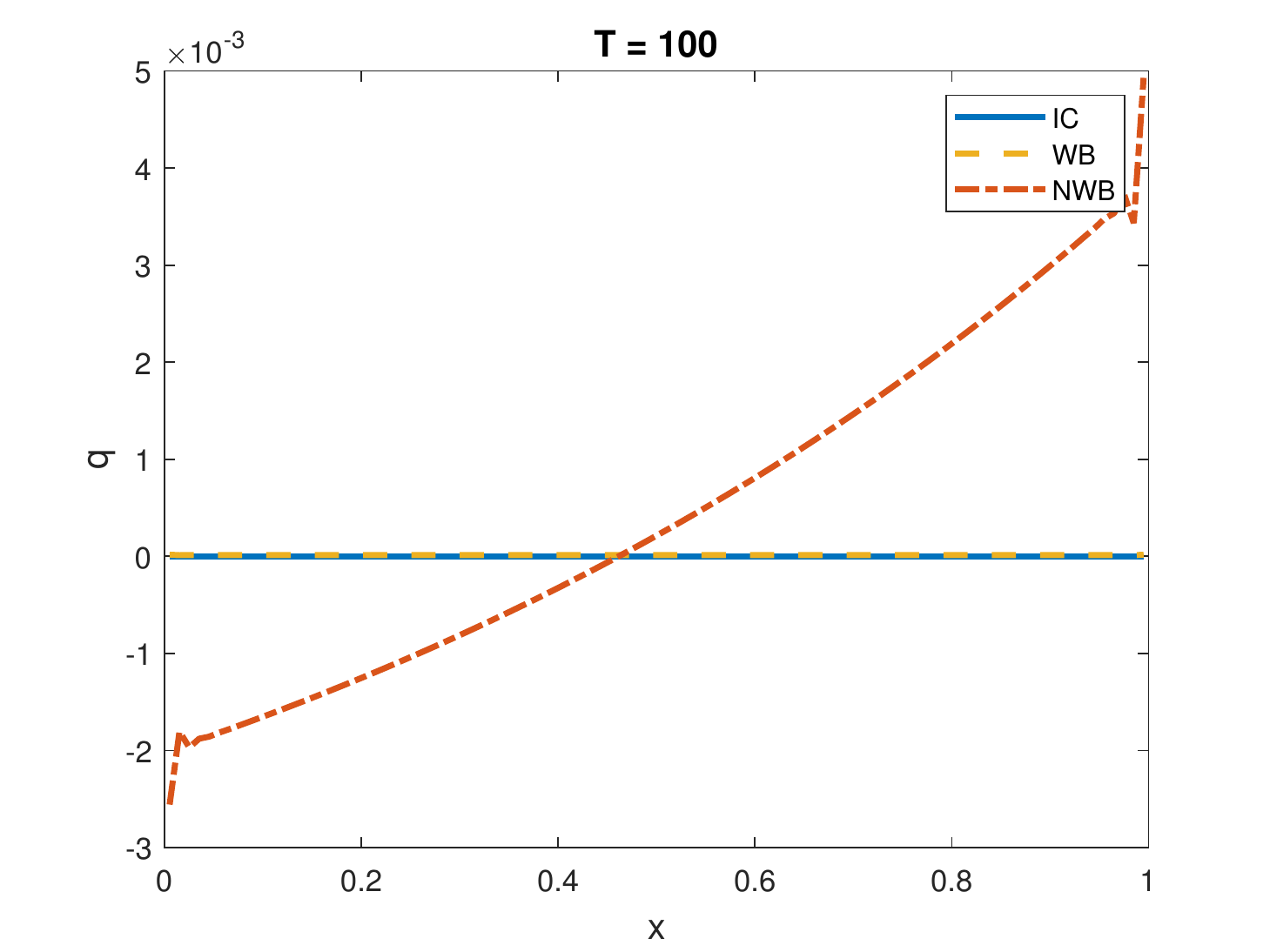}
  \caption{}
\end{subfigure}
\caption{Isothermal test: long time solution profiles when $\veps =
  0.01$ for (A) $\rho$ and (B) $q$.}  
\label{fig:lt_eps10-2}
\end{figure}
\begin{figure}[h]
\centering
\begin{subfigure}{.4\textwidth}
  \centering
  \includegraphics[width=\linewidth]{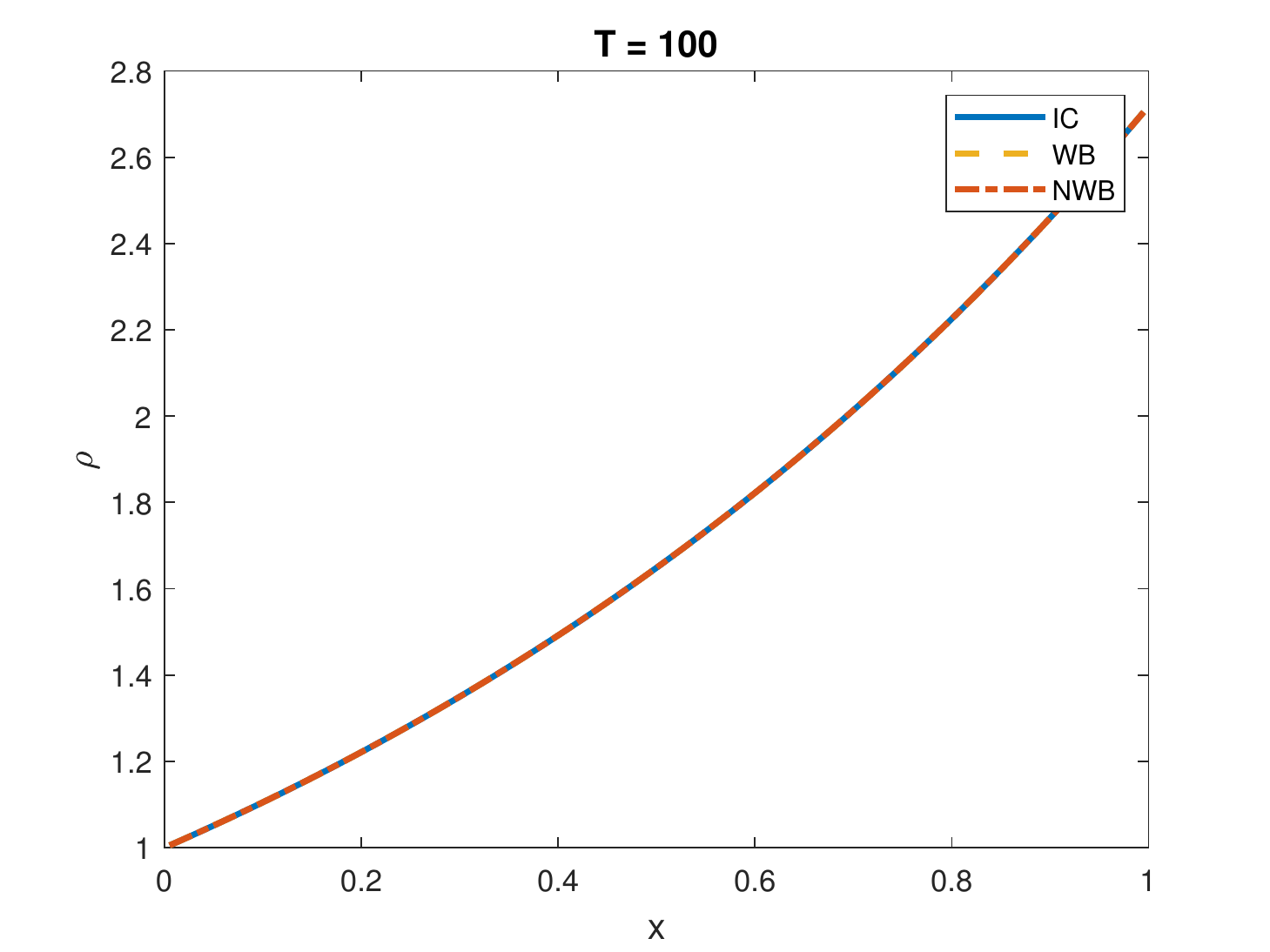}
  \caption{}
\end{subfigure}%
\begin{subfigure}{.4\textwidth}
  \centering
  \includegraphics[width=\linewidth]{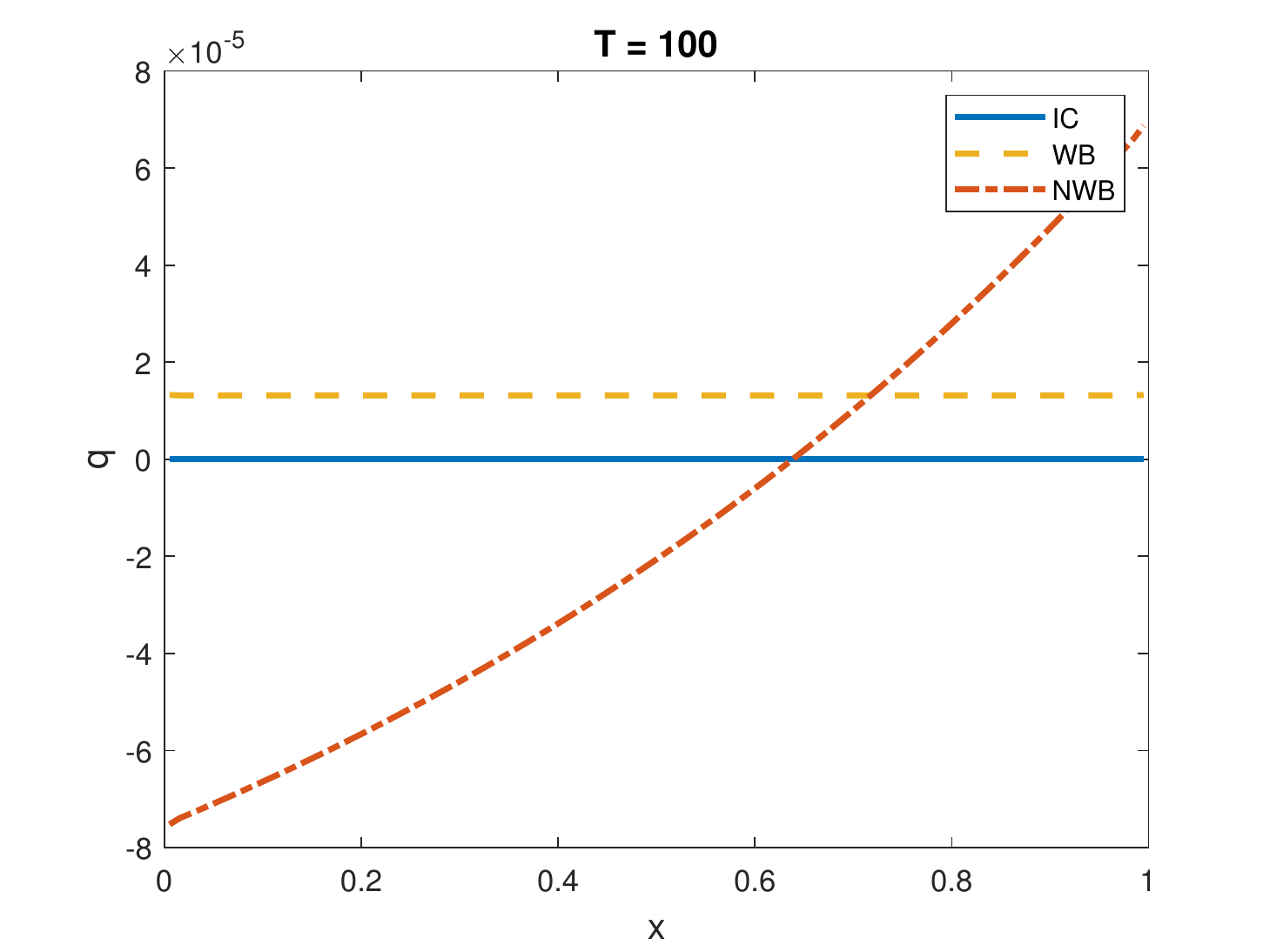}
  \caption{}
\end{subfigure}
\caption{Isothermal test: long time solution profiles when $\veps =
  0.001$ for (A) $\rho$ and (B) $q$.}  
\label{fig:lt_eps10-3}
\end{figure} 

As was noted before, in the asymptotic regime, the non well-balanced
AP scheme also exudes the well-balancing property, and hence
it can be seen from Figure \ref{fig:lt_eps10-3} that the results are
comparable with that of the well balanced scheme when
$\veps=0.001$. The oscillations that can be seen for intermediate
values of the stiffness parameter ($\veps = 0.1, 0.01$) in Figures
\ref{fig:lt_eps10-1}-\ref{fig:lt_eps10-2} for the non well-balanced
scheme can be attributed to the loss of accuracy for the IMEX scheme
in this range of $\veps$; see e.g.\ \cite{PR01} for related
discussions. From the three above-mentioned test cases, it is evident
that the scheme preserves the isothermal hydrostatic solution.  

\subsubsection{Isentropic Hydrostatic Solution}
\label{sec:is_hydro}

In this test case, we solve the system which is initially in
isentropic hydrostatic equilibrium for $\beta = 1$. The initial data
reads 
\begin{equation}
  \label{eq:isentrop}
  \rho(0,x) \equiv \rho_e(x) = \left(1
    +\frac{\gamma-1}{\gamma}\phi(x)\right)^\frac{1}{\gamma -1}, \quad
  u(0,x) \equiv u_e(x) = 0.
\end{equation}
The domain is $[0,1]$ and the boundary conditions are
interpolation. We also take the specific heat ratio to be $\gamma =
1.4$.

First, for a range of $\veps$, we test the $L^1$ error of the
scheme for different potential functions and for different mesh
sizes. The simulations are run for a final time of $T=2.$ Results are
presented in Table \ref{Tab:isentrop_err}. Similar to the isothermal
hydrostatic test case, here too we observe good precision in preserving
the steady state solution. 
\begin{table}[h]
\begin{subtable}{.51\linewidth}
\small
\begin{tabular}{cccc}
\hline
\multicolumn{4}{|c|}{$\veps = 1$}                                                                                                                      \\ \hline
\multicolumn{1}{|c|}{$\phi$}                          & \multicolumn{1}{c|}{N}    & \multicolumn{1}{c|}{Error in $\rho$} & \multicolumn{1}{c|}{Error in $q$} \\ \hline
\multicolumn{1}{|c|}{\multirow{2}{*}{$x$}}            & \multicolumn{1}{c|}{100}  & \multicolumn{1}{c|}{3.0399E-07}      & \multicolumn{1}{c|}{7.6093E-07}   \\ \cline{2-4} 
\multicolumn{1}{|c|}{}                                & \multicolumn{1}{c|}{1000} & \multicolumn{1}{c|}{3.0717E-09}      & \multicolumn{1}{c|}{7.4895E-09}   \\ \hline
\multicolumn{1}{|c|}{\multirow{2}{*}{$x^2/2$}}        & \multicolumn{1}{c|}{100}  & \multicolumn{1}{c|}{1.5297E-07}      & \multicolumn{1}{c|}{1.8205E-07}   \\ \cline{2-4} 
\multicolumn{1}{|c|}{}                                & \multicolumn{1}{c|}{1000} & \multicolumn{1}{c|}{1.4429E-09}      & \multicolumn{1}{c|}{1.7516E-09}   \\ \hline
\multicolumn{1}{|c|}{\multirow{2}{*}{sin($2 \pi x$)}} & \multicolumn{1}{c|}{100}  & \multicolumn{1}{c|}{3.5858E-05}      & \multicolumn{1}{c|}{1.6335E-06}   \\ \cline{2-4} 
\multicolumn{1}{|c|}{}                                & \multicolumn{1}{c|}{1000} & \multicolumn{1}{c|}{3.5668E-07}      & \multicolumn{1}{c|}{3.4458E-09}   \\ \hline
\end{tabular}
\end{subtable}
\begin{subtable}{.48\linewidth}
\small
\begin{tabular}{llll}
\hline
\multicolumn{4}{|c|}{$\veps = 0.1$}                                                                                                                    \\ \hline
\multicolumn{1}{|c|}{$\phi$}                          & \multicolumn{1}{c|}{N}    & \multicolumn{1}{c|}{Error in $\rho$} & \multicolumn{1}{c|}{Error in $q$} \\ \hline
\multicolumn{1}{|c|}{\multirow{2}{*}{$x$}}            & \multicolumn{1}{c|}{100}  & \multicolumn{1}{c|}{8.8097E-08}      & \multicolumn{1}{c|}{2.2017E-06}   \\ \cline{2-4} 
\multicolumn{1}{|c|}{}                                & \multicolumn{1}{c|}{1000} & \multicolumn{1}{c|}{8.9517E-10}      & \multicolumn{1}{c|}{2.1593E-08}   \\ \hline
\multicolumn{1}{|c|}{\multirow{2}{*}{$x^2/2$}}        & \multicolumn{1}{c|}{100}  & \multicolumn{1}{c|}{1.3723E-07}      & \multicolumn{1}{c|}{5.0076E-07}   \\ \cline{2-4} 
\multicolumn{1}{|c|}{}                                & \multicolumn{1}{c|}{1000} & \multicolumn{1}{c|}{1.2993E-09}      & \multicolumn{1}{c|}{4.8594E-09}   \\ \hline
\multicolumn{1}{|c|}{\multirow{2}{*}{sin($2 \pi x$)}} & \multicolumn{1}{c|}{100}  & \multicolumn{1}{c|}{3.6382E-05}      & \multicolumn{1}{c|}{4.4131E-06}   \\ \cline{2-4} 
\multicolumn{1}{|c|}{}                                & \multicolumn{1}{c|}{1000} & \multicolumn{1}{c|}{3.5705E-07}      & \multicolumn{1}{c|}{4.4168E-09}   \\ \hline
\end{tabular}
\end{subtable}%
\vspace{5mm}

\begin{subtable}{.51\linewidth}
\small
\begin{tabular}{llll}
\hline
\multicolumn{4}{|c|}{$\veps = 0.01$}                                                                                                                   \\ \hline
\multicolumn{1}{|c|}{$\phi$}                          & \multicolumn{1}{c|}{N}    & \multicolumn{1}{c|}{Error in $\rho$} & \multicolumn{1}{c|}{Error in $q$} \\ \hline
\multicolumn{1}{|c|}{\multirow{2}{*}{$x$}}            & \multicolumn{1}{c|}{100}  & \multicolumn{1}{c|}{2.2198E-08}      & \multicolumn{1}{c|}{2.6719E-06}   \\ \cline{2-4} 
\multicolumn{1}{|c|}{}                                & \multicolumn{1}{c|}{1000} & \multicolumn{1}{c|}{2.0347E-10}      & \multicolumn{1}{c|}{2.6542E-08}   \\ \hline
\multicolumn{1}{|c|}{\multirow{2}{*}{$x^2/2$}}        & \multicolumn{1}{c|}{100}  & \multicolumn{1}{c|}{1.3446E-07}      & \multicolumn{1}{c|}{5.9381E-07}   \\ \cline{2-4} 
\multicolumn{1}{|c|}{}                                & \multicolumn{1}{c|}{1000} & \multicolumn{1}{c|}{1.2895E-09}      & \multicolumn{1}{c|}{5.8580E-09}   \\ \hline
\multicolumn{1}{|c|}{\multirow{2}{*}{sin($2 \pi x$)}} & \multicolumn{1}{c|}{100}  & \multicolumn{1}{l|}{3.6564E-05}      & \multicolumn{1}{l|}{5.1607E-06}   \\ \cline{2-4} 
\multicolumn{1}{|c|}{}                                & \multicolumn{1}{c|}{1000} & \multicolumn{1}{l|}{3.5718E-07}      & \multicolumn{1}{l|}{5.1886E-09}   \\ \hline
\end{tabular}
\end{subtable}
\begin{subtable}{.48\linewidth}
\small
\begin{tabular}{llll}
\hline
\multicolumn{4}{|c|}{$\veps = 0.001$}                                                                                                                  \\ \hline
\multicolumn{1}{|c|}{$\phi$}                          & \multicolumn{1}{c|}{N}    & \multicolumn{1}{c|}{Error in $\rho$} & \multicolumn{1}{c|}{Error in $q$} \\ \hline
\multicolumn{1}{|c|}{\multirow{2}{*}{$x$}}            & \multicolumn{1}{c|}{100}  & \multicolumn{1}{c|}{1.8445E-08}      & \multicolumn{1}{c|}{2.7065E-06}   \\ \cline{2-4} 
\multicolumn{1}{|c|}{}                                & \multicolumn{1}{c|}{1000} & \multicolumn{1}{c|}{1.6062E-10}      & \multicolumn{1}{c|}{2.7060E-08}   \\ \hline
\multicolumn{1}{|c|}{\multirow{2}{*}{$x^2/2$}}        & \multicolumn{1}{c|}{100}  & \multicolumn{1}{c|}{1.3427E-07}      & \multicolumn{1}{c|}{5.9992E-07}   \\ \cline{2-4} 
\multicolumn{1}{|c|}{}                                & \multicolumn{1}{c|}{1000} & \multicolumn{1}{c|}{1.2840E-09}      & \multicolumn{1}{c|}{5.9572E-09}   \\ \hline
\multicolumn{1}{|c|}{\multirow{2}{*}{sin($2 \pi x$)}} & \multicolumn{1}{c|}{100}  & \multicolumn{1}{l|}{3.6580E-05}      & \multicolumn{1}{l|}{5.2108E-06}   \\ \cline{2-4} 
\multicolumn{1}{|c|}{}                                & \multicolumn{1}{c|}{1000} & \multicolumn{1}{l|}{3.5719E-07}      & \multicolumn{1}{l|}{5.2778E-09}   \\ \hline
\end{tabular}
\end{subtable}
\caption{Errors in the density $\rho$ and the momentum $q$ for
  different potentials and for a range of $\veps$ using different mesh sizes in the
  isentropic test case.} 
\label{Tab:isentrop_err}
\end{table}

Next, we test the efficacy of the scheme in simulating the evolution of
small density perturbations added to the initial data, i.e.\ 
\begin{equation*}
  \rho(0,x) =  \left(1
    +\frac{\gamma-1}{\gamma}\phi(x)\right)^\frac{1}{\gamma -1} + \zeta
  \text{exp}(-100(x-0.5)^2). 
\end{equation*}

Simulations are performed upto a final time $T = 0.25$ on a mesh of 100
cells with interpolation boundary conditions for $\zeta = 10^{-3},
10^{-5}$. The results are presented in Figure \ref{fig:is_pert_eps1}
for the non-stiff regime. Evidently the well-balanced scheme is able to
resolve the solution much better than the non well-balanced scheme for
both high and low values of $\veps$. In the stiff regime, the
non well-balanced scheme's performance improves due to its AP property
as is evident from Figure \ref{fig:is_pert_eps10-3}. 

Thus in this subsection, we have shown that the scheme performs well in
maintaining the steady states of the Euler system for both isothermal
and isentropic equations of state with very good accuracy. 

\begin{figure}
\centering
\begin{subfigure}{.32\textwidth}
  \centering
  \includegraphics[width=\linewidth]{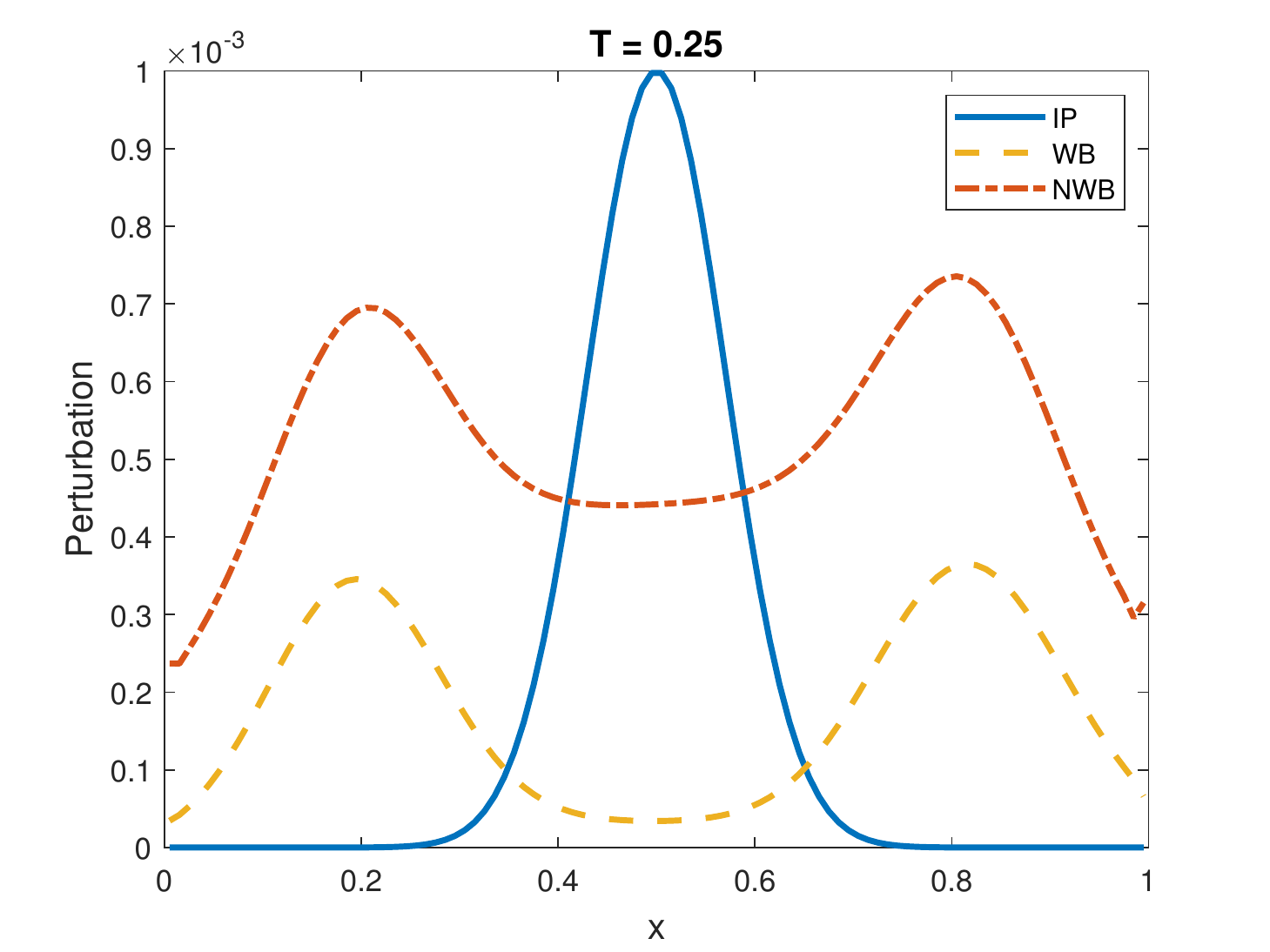}
  \caption{}
\end{subfigure}%
\begin{subfigure}{.32\textwidth}
  \centering
  \includegraphics[width=\linewidth]{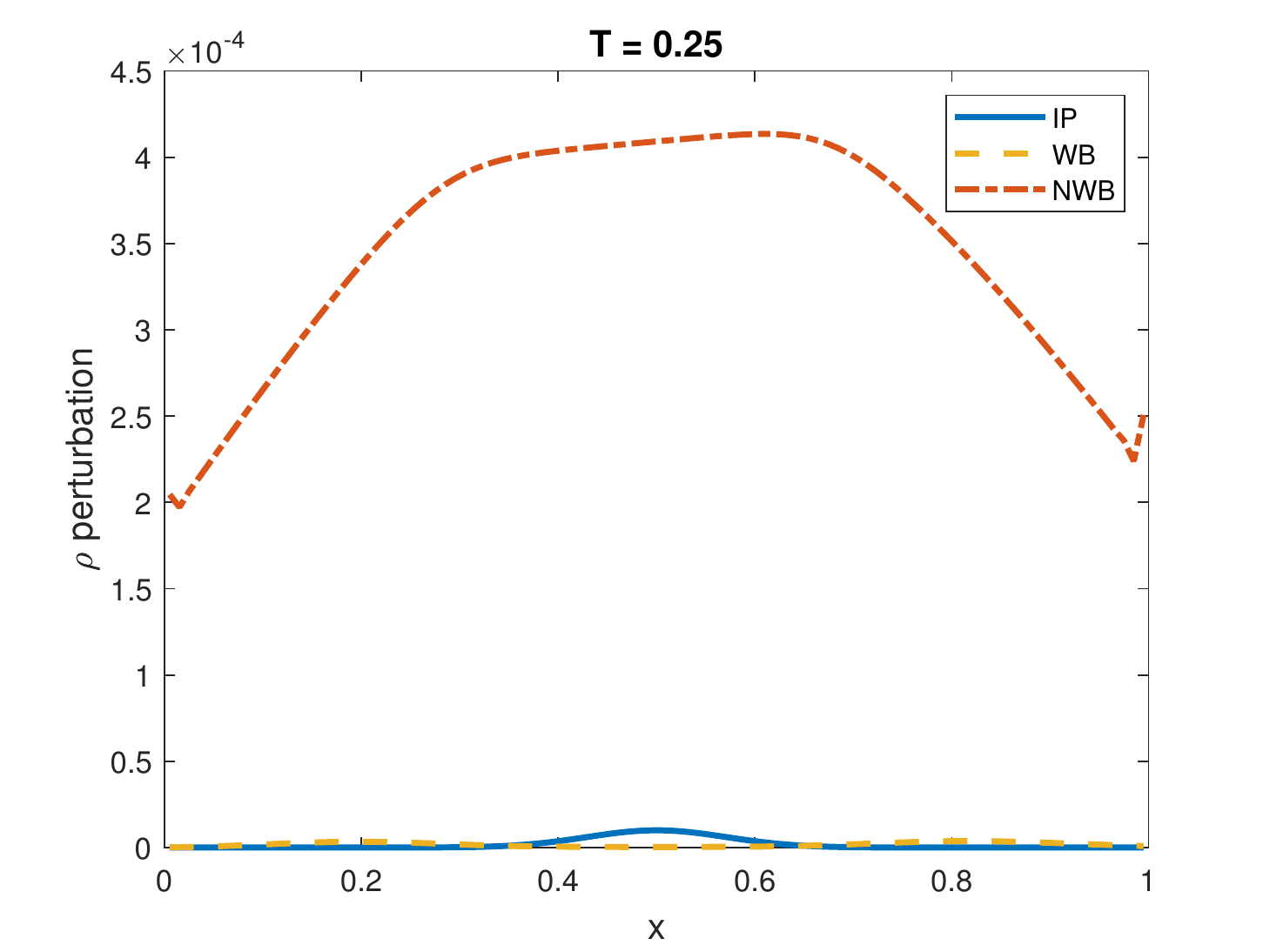}
  \caption{}
\end{subfigure}
\begin{subfigure}{.32\textwidth}
  \centering
  \includegraphics[width=\linewidth]{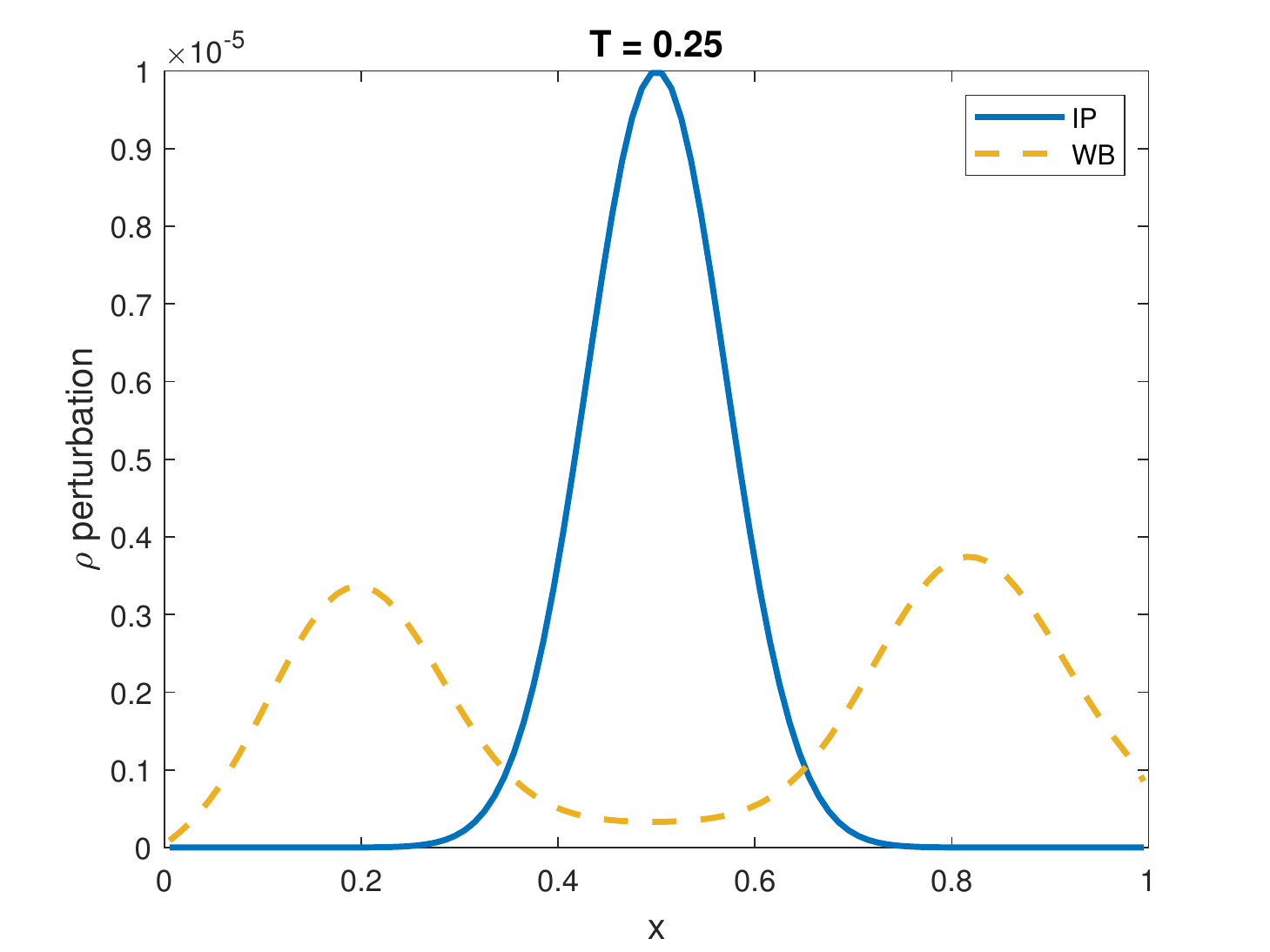}
  \caption{}
\end{subfigure}
\caption{Isentropic test: evolution in the density perturbation in the
  non-stiff regime for $\veps=1$. (a) $\zeta = 10^{-3}$ (b)
  $\zeta = 10^{-5}$; comparison of well-balanced scheme with non-well
  balanced (c) $\zeta = 10^{-5}$; comparison of well-balanced scheme
  with initial perturbation.}   
\label{fig:is_pert_eps1}
\end{figure}
\begin{figure}[h]
\centering
\begin{subfigure}{.4\textwidth}
  \centering
  \includegraphics[width=\linewidth]{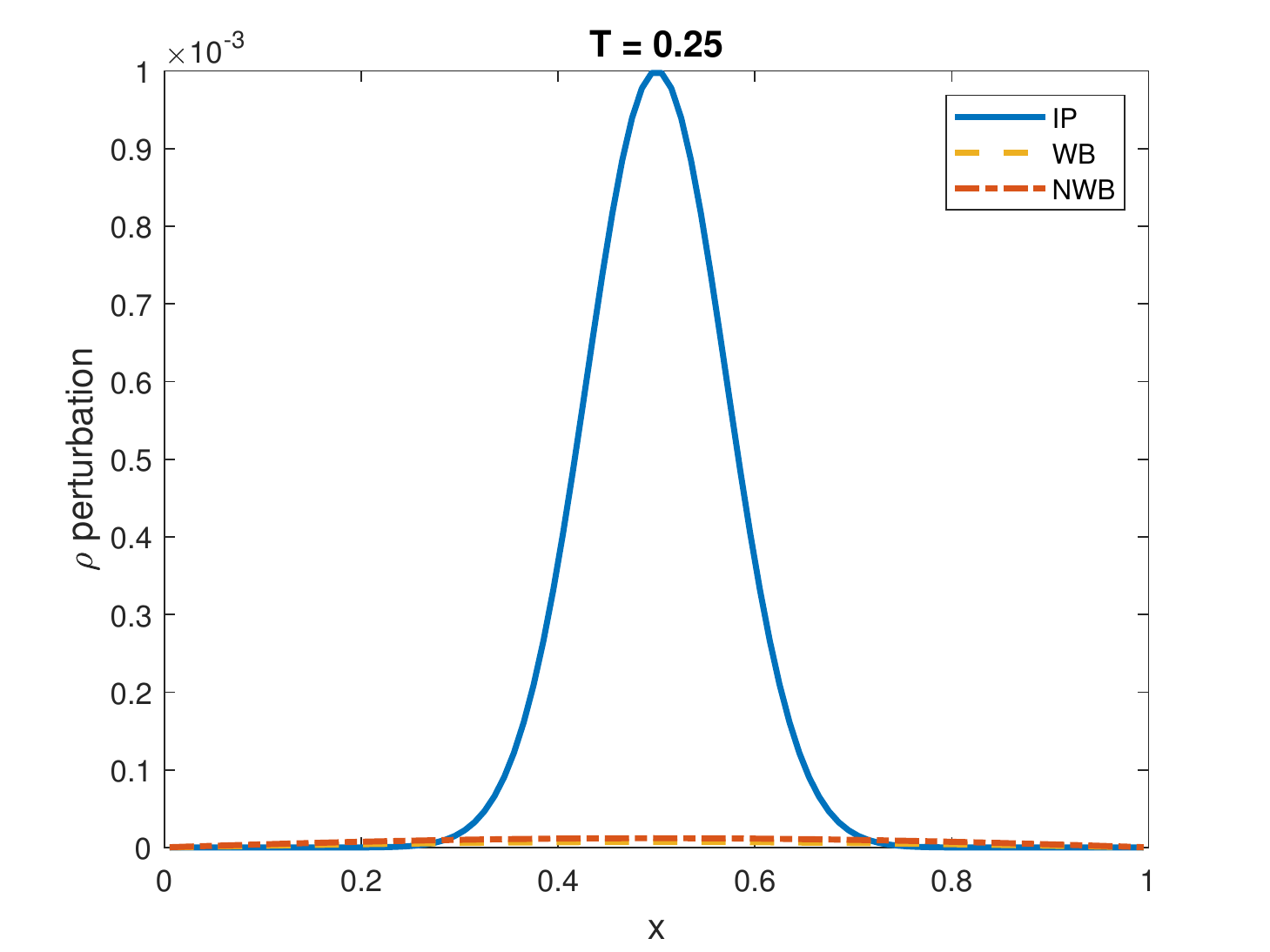}
  \caption{}
\end{subfigure}%
\begin{subfigure}{.4\textwidth}
  \centering
  \includegraphics[width=\linewidth]{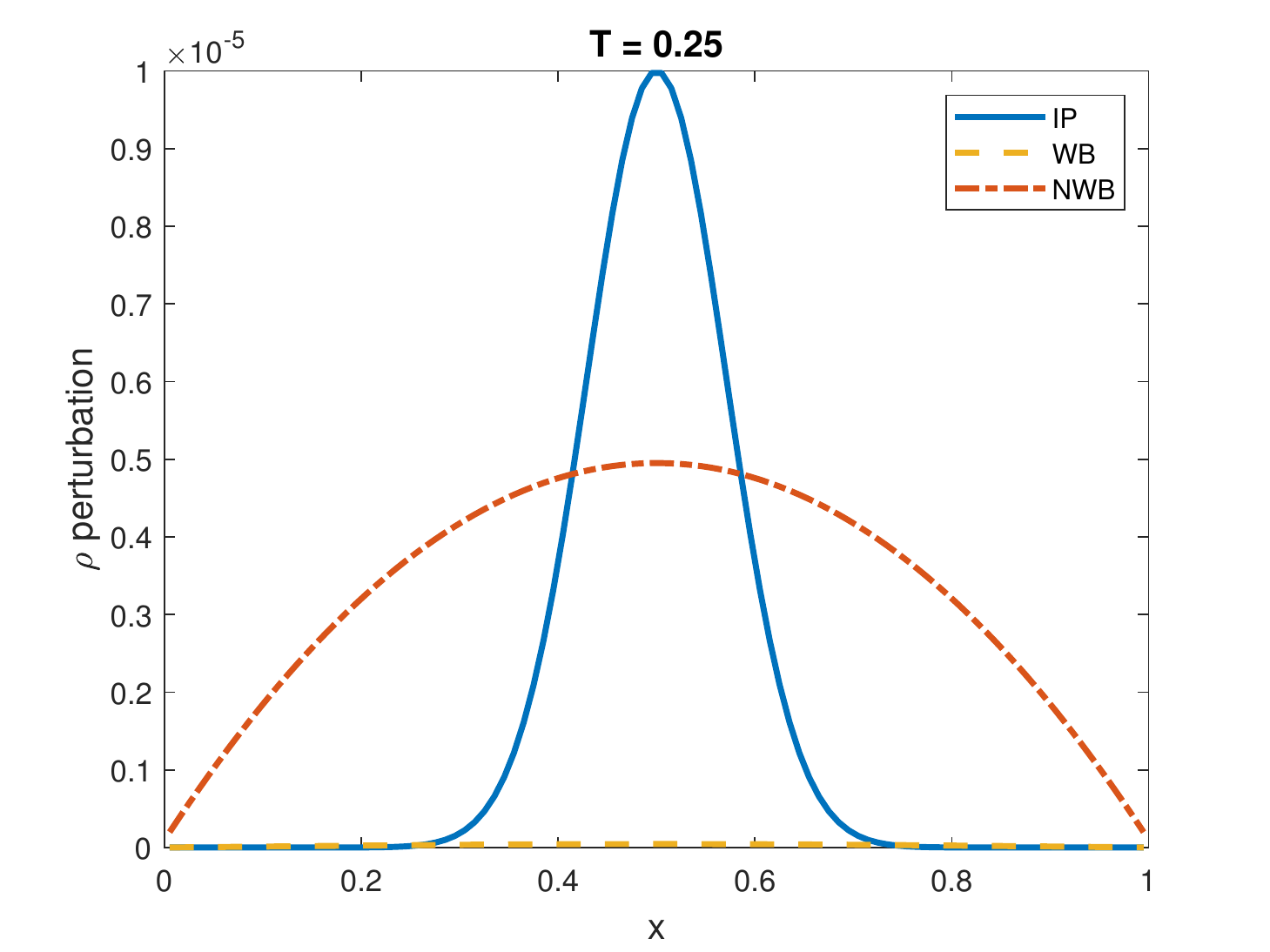}
  \caption{}
\end{subfigure}
\caption{Isentropic test: evolution in the density perturbation in the
  stiff regime for $\veps=0.001$ for (a) $\zeta = 10^{-3}$ (b) $\zeta =
  10^{-5}$.}  
\label{fig:is_pert_eps10-3}
\end{figure}

\subsection{Sensitivity to Mesh Size}
\label{sec:sensitivity}

In this test, our aim is to study the dependence of the accuracy of
the scheme on mesh sizes in the asymptotic regime ($\veps = 0.001$) by 
comparing it with a non AP scheme. For this purpose, the initial data
as given in \cite{CCG10} which is a centered arch function and it
reads 
\begin{equation}
  \label{eq:ChalonsIC}
  (\rho, u) = 
  \begin{cases}
    (1,0), & \text{if}\ -0.2 < x < 0.2, \\
    (2,0), &  \text{otherwise},
  \end{cases}
\end{equation}
in the domain $[-0.5,0.5]$ with periodic boundary conditions and
$\lambda_\CFL = 0.45$. The test is carried out for both hyperbolic
and parabolic relaxations. Figure \ref{fig:para_AP} shows that the AP
scheme's performance does not show any dependence on the mesh sizes,
giving almost identical results for under-resolved ($\Delta x = 0.01$),
resolved ($\Delta x = 0.001$), and over-resolved ($\Delta x = 0.0001$)
meshes in the parabolic regime. However, the non AP scheme blows up
for the under-resolved mesh (Figure \ref{fig:para_NAP_blowup}), but it
shows results similar to the AP scheme for the resolved and
over-resolved meshes as can be seen from Figure
\ref{fig:para_NAP}. Similar conclusions can again be drawn in the case
of hyperbolic relaxation. Results are presented in Figure
\ref{fig:sens_mesh_hyp}. Even though a blowup is not observed for the
non AP scheme, oscillations can be seen at the discontinuities on the
under-resolved mesh. 

\begin{figure}[h]
\centering
\begin{subfigure}{.33\textwidth}
  \centering
  \includegraphics[width=\linewidth]{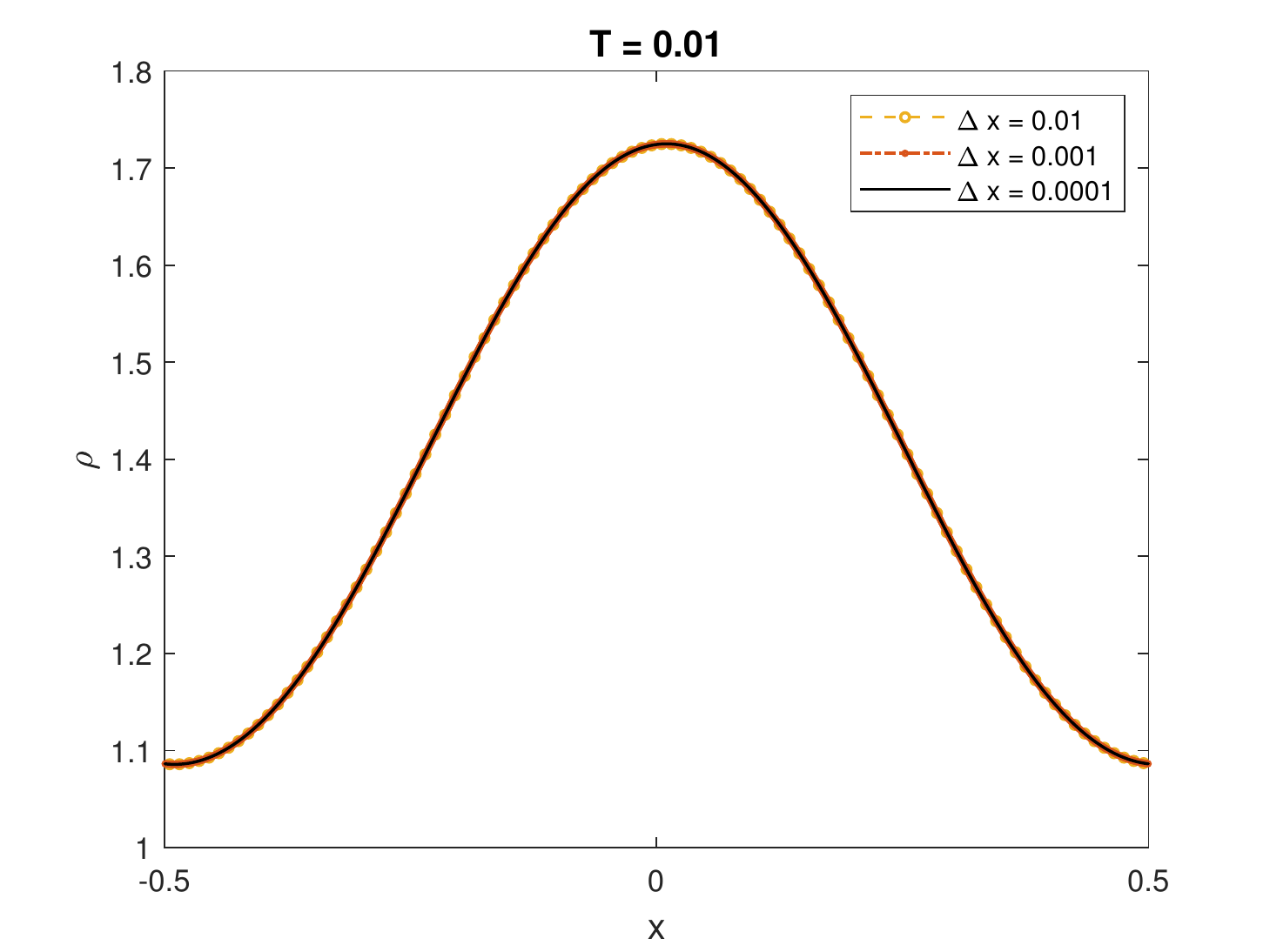}
  \caption{}
  \label{fig:para_AP}
\end{subfigure}%
\begin{subfigure}{.33\textwidth}
  \centering
  \includegraphics[width=\linewidth]{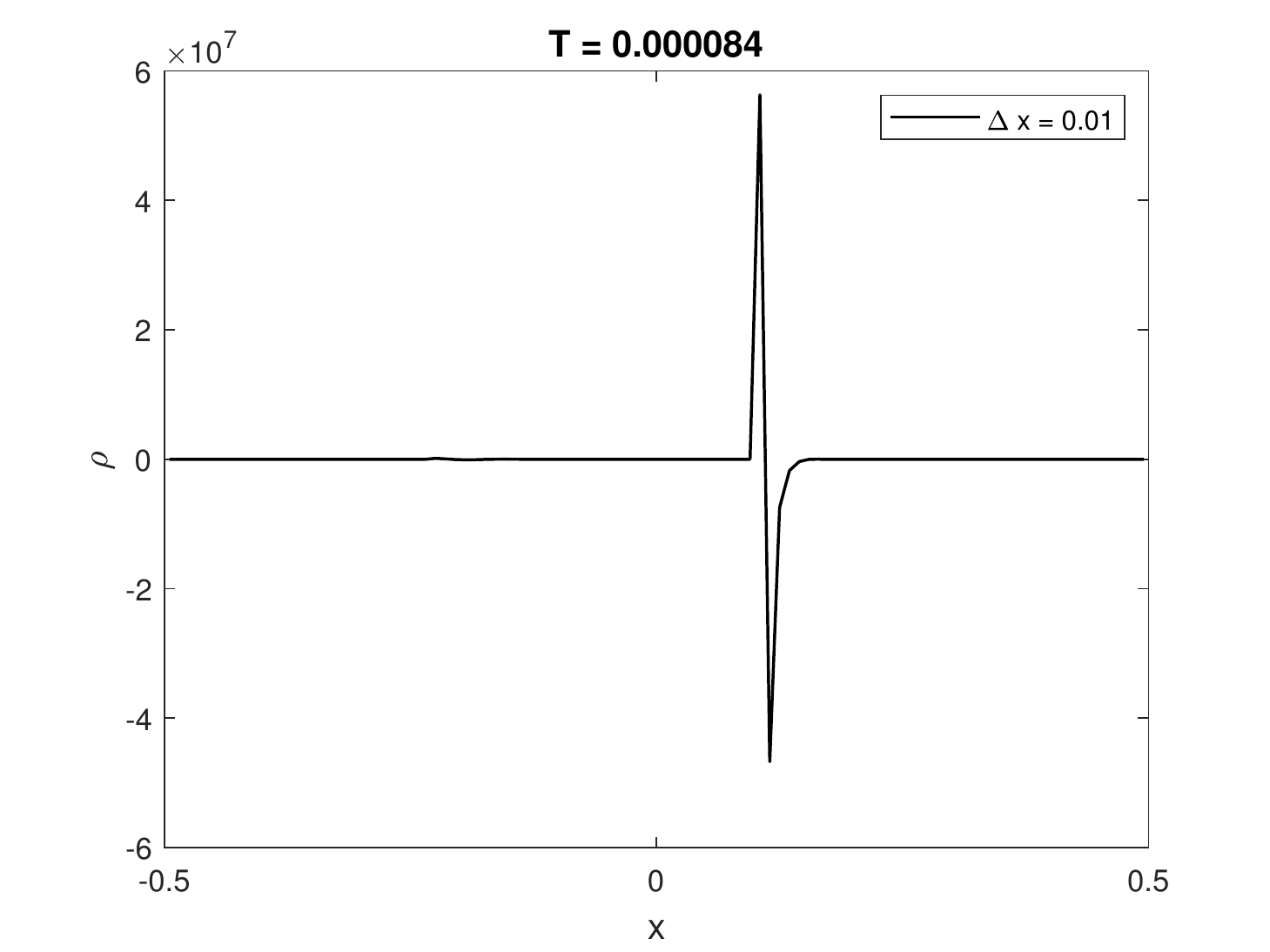}
  \caption{}
 \label{fig:para_NAP_blowup}
\end{subfigure}
\begin{subfigure}{.33\textwidth}
  \centering
  \includegraphics[width=\linewidth]{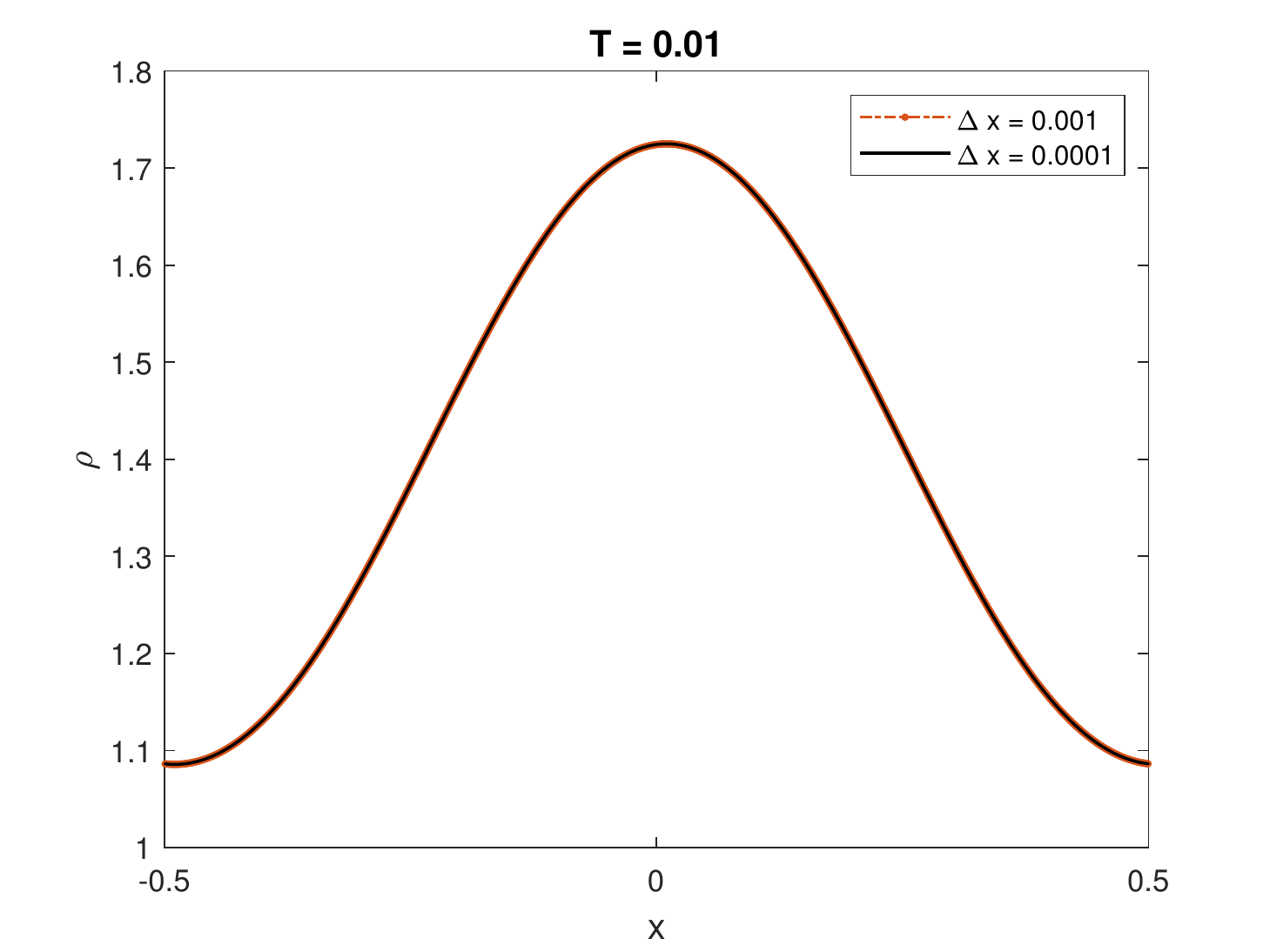}
  \caption{}
  \label{fig:para_NAP}
\end{subfigure}
\caption{Mesh sensitivity test: solution profile of $\rho$ for
  different mesh sizes for $\veps = 0.001$ (A) AP scheme (B) non
  AP scheme, $\Delta x = 0.01$ (C) non AP scheme, $\Delta x = 0.001,
  0.0001$.}  
\label{fig:sens_mesh_para}
\end{figure}

\begin{figure}[htbp]
\centering
\begin{subfigure}{.4\textwidth}
  \centering
  \includegraphics[width=\linewidth]{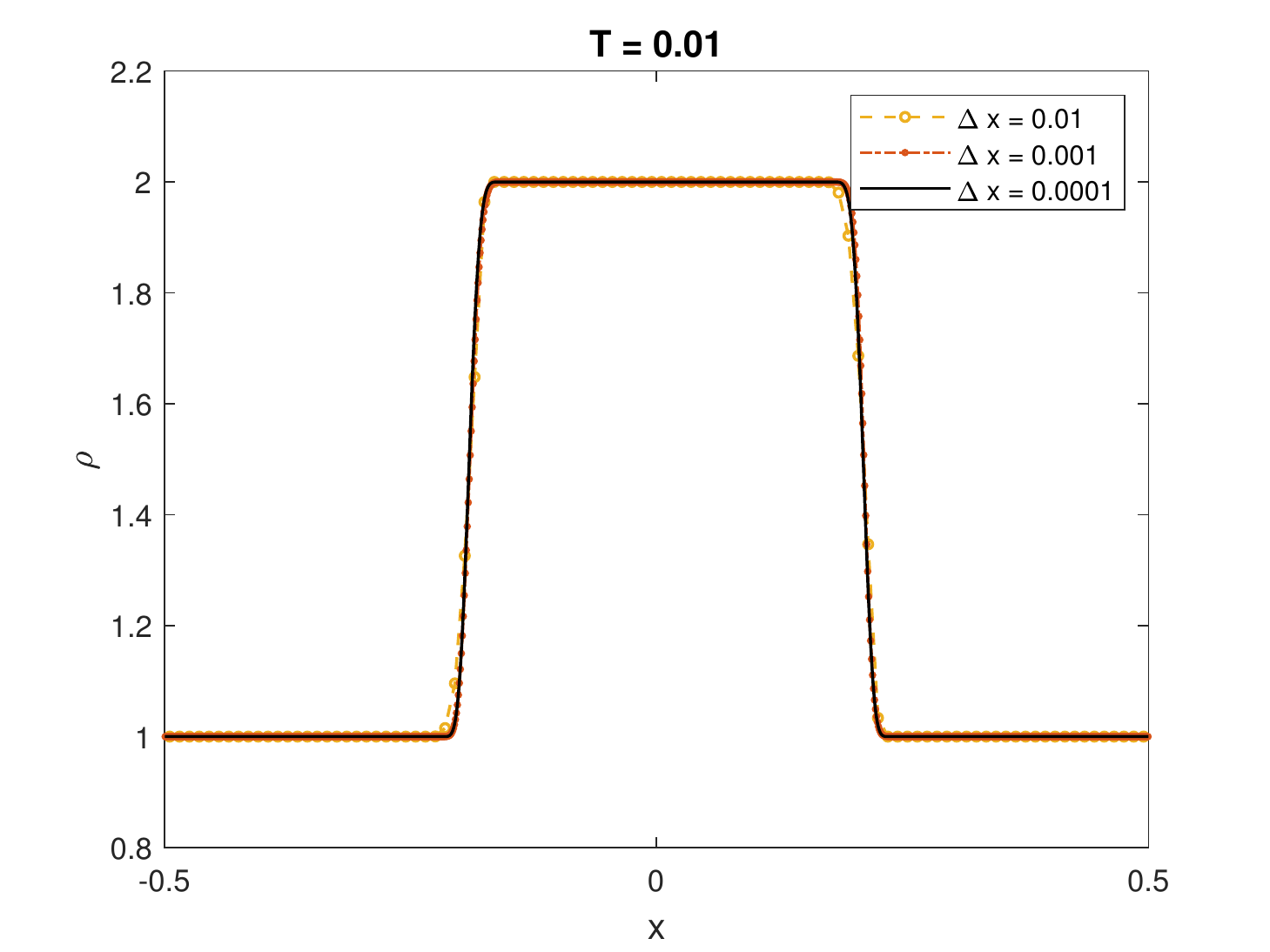}
  \caption{}
  \label{fig:hyp_AP}
\end{subfigure}%
\begin{subfigure}{.4\textwidth}
  \centering
  \includegraphics[width=\linewidth]{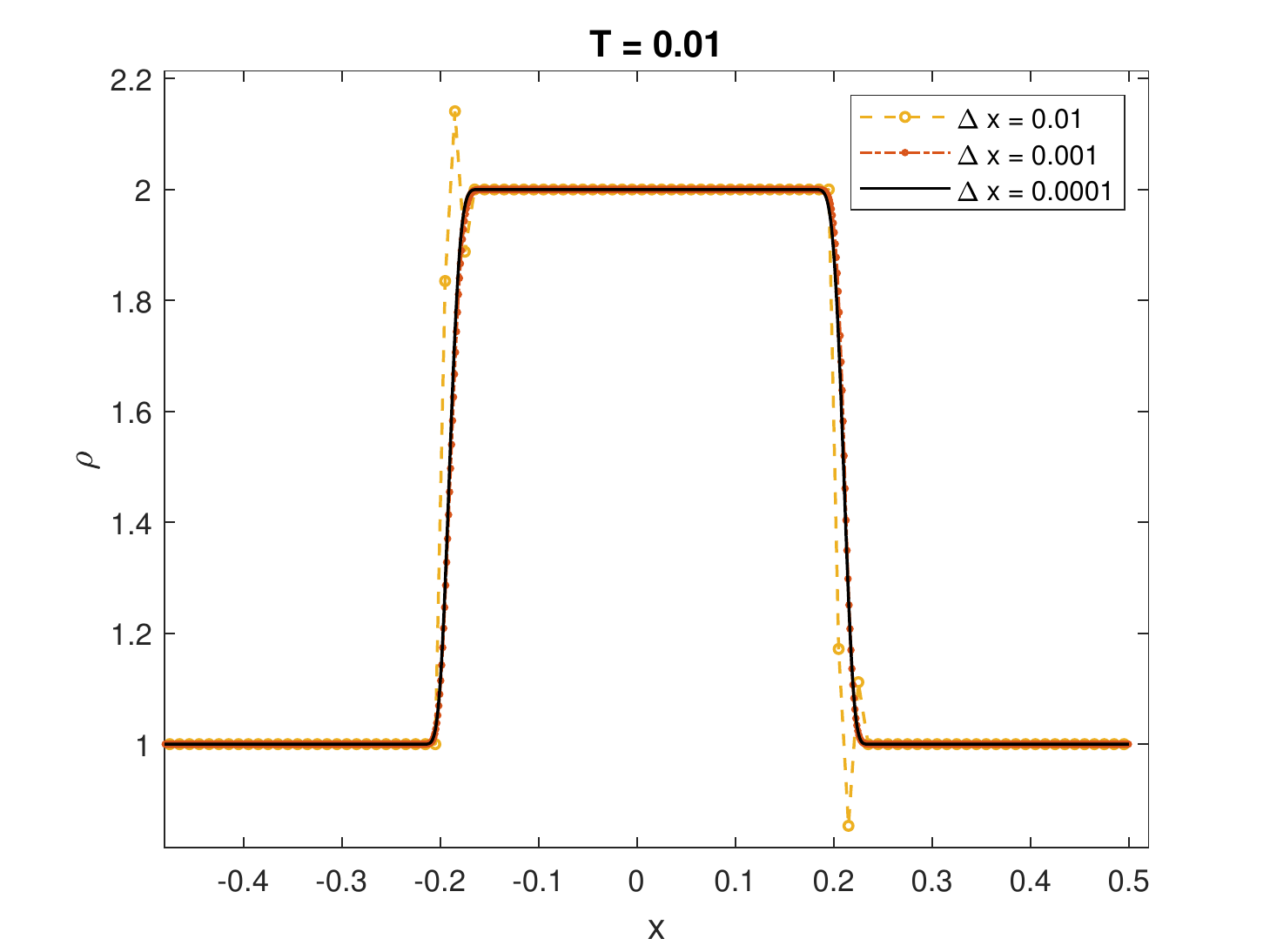}
  \caption{}
 \label{fig:hyp_NAP}
\end{subfigure}
\caption{Mesh sensitivity test: solution profile of $\rho$ for
  different mesh sizes for different mesh sizes when $\veps =
  0.001$ (A) AP scheme (B) non AP scheme.}  
\label{fig:sens_mesh_hyp}
\end{figure}

\section{Concluding Remarks}
\label{sec:con}

We have designed and analysed a unified AP scheme which captures both
the hyperbolic and parabolic limits of the Euler system with gravity
and friction. The time semi-discrete and semi-implicit scheme is based on
the ideas presented in \cite{BPR17}. A reformulation of the
semi-implicit scheme admits a fully-explicit formulation which is
stable under a parabolic CFL condition. Though for the hyperbolic
relaxation the time steps are dependent on $\veps$, the CFL
restriction does not degrade; rather it becomes less and less severe
as $\veps\to0$. A fully-discrete scheme is obtained using a finite
volume treatment which makes use of an equilibrium reconstruction of the 
interface values, source term upwinding and a gingerly choice of
central discretisation for the mass update. Both the semi-discrete and
space-time fully-discrete scheme are shown to be AP for both the
hyperbolic and parabolic limit equations. Furthermore, the
fully-discrete scheme is shown to well-balanced for hydrostatic steady
states. The numerical case studies presented clearly demonstrate
the AP and well-balancing properties of the developed scheme. They
also showcase the superiority of the designed scheme over its non
well-balanced and non AP counterparts. In conclusion, the aim of
developing a unified AP and well-balanced scheme is achieved and is
justified through the material presented in the paper.    

\bibliographystyle{abbrv}

\end{document}